\documentclass{amsart}

 \usepackage{color}
\usepackage{amsmath,amssymb,amsfonts,amsthm,bm}
\usepackage{graphicx}
\usepackage{enumerate}
 \newtheorem{theorem}{Theorem}[section]
 
 \newtheorem{lemma}[theorem]{Lemma}
 \newtheorem{proposition}[theorem]{Proposition}
 \theoremstyle{assumption}
   \newtheorem{assumption}[theorem]{Assumption}
\theoremstyle{definition}
\newtheorem{definition}[theorem]{Definition}
 \theoremstyle{remark}
\newtheorem{remark}[theorem]{Remark}
\numberwithin{equation}{section}

 \newcommand{\eps}{\varepsilon}

\newcommand{\norm}[1]{\Vert#1\Vert}
 \newcommand{\abs}[1]{\left\vert#1\right\vert}

 \newcommand{\inner}[1]{\left(#1\right)}
\newcommand{\comi}[1]{\left<#1\right>}

 \newcommand{\normm}[1]{{ \vert\kern-0.25ex \vert\kern-0.25ex \vert #1 
     \vert\kern-0.25ex \vert\kern-0.25ex \vert}}

\makeatletter

\def\@startsection#1#2#3#4#5#6{%
 \if@noskipsec \leavevmode \fi
 \par \@tempskipa #4\relax
 \@afterindentfalse
 \ifdim \@tempskipa <\z@ \@tempskipa -\@tempskipa \@afterindentfalse\fi
 \if@nobreak \everypar{}\else
     \addpenalty\@secpenalty\addvspace\@tempskipa\fi
 \@ifstar{\@dblarg{\@sect{#1}{\@m}{#3}{#4}{#5}{#6}}}%
         {\@dblarg{\@sect{#1}{#2}{#3}{#4}{#5}{#6}}}%
}

\def\@settitle{%
  \bgroup
  \centering
  \vglue1cm
  \fontsize{12}{15}\fontseries{b}\selectfont
  \uppercasenonmath\@title
  \@title
  \vskip20pt plus 6pt minus 8pt
  \egroup
}

\def\@setauthors{%
  \begingroup
  \trivlist
  \centering \bfseries
 \normalsize\@topsep30\p@\relax
  \advance\@topsep by -\baselineskip
  \item\relax
  \andify\authors
 {\rmfamily\authors}%
  \endtrivlist
  \endgroup
}

\def\@setaddresses{\par
  \nobreak \begingroup
\normalsize
  \def\author##1{\nobreak\addvspace\bigskipamount}%
  \def\\{\unskip, \ignorespaces}%
  \interlinepenalty\@M
  \def\address##1##2{\begingroup
    \par\addvspace\bigskipamount\noindent
    \@ifnotempty{##1}{(\ignorespaces##1\unskip) }%
    {\ignorespaces##2}\par\endgroup}%
  \def\curraddr##1##2{\begingroup
    \@ifnotempty{##2}{\nobreak\indent{\itshape Current address}%
      \@ifnotempty{##1}{, \ignorespaces##1\unskip}\/:\space
      ##2\par}\endgroup}%
  \def\email##1##2{\begingroup
    \@ifnotempty{##2}{\nobreak\noindent{\itshape E-mail address}%
      \@ifnotempty{##1}{, \ignorespaces##1\unskip}\/: 
       ##2\par}\endgroup}%
   \def\urladdr##1##2{\begingroup
    \@ifnotempty{##2}{\nobreak\indent{\itshape URL}%
      \@ifnotempty{##1}{, \ignorespaces##1\unskip}\/:\space
      \ttfamily##2\par}\endgroup}%
  \addresses
  \endgroup
}

 \renewcommand\section{\@startsection{section}{1}{\z@}%
{27pt plus 6pt minus 8pt}{14pt plus 6pt minus 8pt}
{\center\normalfont\large\bfseries}}
\renewcommand\subsection{\@startsection{subsection}{2}{\z@}%
{27pt plus 6pt minus 8pt}{14pt plus 6pt minus 8pt}
{ \center
\normalfont\bfseries}}

\def\subsubsection{\@startsection{subsubsection}{3}%
  \z@{.5\linespacing\@plus.7\linespacing}{-.5em}%
  {\normalfont\itshape}}

 \def\@mainsize{10}\def\@ptsize{0}%
  \def\@typesizes{%
    \or{5}{6}\or{6}{7}\or{7}{8}\or{8}{9}\or{9}{11}%
    \or{10}{13}
    \or{\@xipt}{13}\or{\@xiipt}{14}\or{\@xivpt}{17}%
    \or{\@xviipt}{20}\or{\@xxpt}{24}}%
  \normalsize \linespacing=\baselineskip

\makeatother

\begin{document}

\title[Well-posedness in Gevrey  space  for 3D  Prandtl  equations]{Well-posedness in Gevrey function space  for 3D   Prandtl  equations without Structural Assumption}

\author[W.-X. Li, N. Masmoudi and T. Yang]{ Wei-Xi Li, Nader Masmoudi \and Tong Yang}

\date{}

\address[W.-X. Li]{School of Mathematics and Statistics,   Wuhan University,  Wuhan 430072, China
\& Hubei Key Laboratory of Computational Science, Wuhan University, Wuhan 430072,  China
  }

\email{
wei-xi.li@whu.edu.cn}

\address[N. Masmoudi]{ Courant Institute of Mathematical Sciences, New  York, USA 
and  Department of Mathematics, NYU Abu-Dhabi, UAE
	}

\email{
	masmoudi@cims.nyu.edu}

\address[T.Yang]{
	Department of Mathematics, City University of Hong Kong, Hong Kong
\&  School of Mathematics and Statistics,   Chongqing University,  Chongqing, China
  }

\email{
matyang@cityu.edu.hk}

\begin{abstract}
We establish the well-posedness  in Gevrey function space with optimal class of regularity  2   for   the   three dimensional Prandtl system without any structural assumption. 
The proof combines  in a novel way   a  new cancellation  in the system with   some of 
the  old ideas  to overcome the difficulty of the  loss of derivatives in the system.
This shows that the three dimensional  instabilities in the system leading to 
  ill-posedness  are not worse than the two  dimensional ones. 
\end{abstract}


 \maketitle


\section{Introduction and main results}

As the foundational system of boundary layer theories, Prandtl equation was 
derived by Prandtl  in 1904 from
the incompressible Navier-Stokes equation with no-slip boundary condition  for the description of  the behavior 
of fluid motion near the boundary when  viscosity vanishes. 
In fact, in this viscous to inviscid limit process,
there exists  a  boundary layer where the majority of the drag experienced by the solid body can be modelled by 
a `simplified'  system derived from the incompressible Navier-Stokes equations
for  balancing the inertial and frictional forces. Outside this layer, the viscosity can be basically
 neglected as it has no significant effect on the fluid so that
 the fluid motion can be modelled by the Euler equation. Even though there are fruitful mathematical theories developed since the seminal works by Oleinik in 1960s, most of the well-posedness
theories are limited to the two space dimensions  under Oleinik's monotonicity condition except the
classical work by Sammartino-Caflisch in 1998 in the framework of analytic functions and some
recent work in Gevrey function spaces.

Prandtl equation can be viewed as a typical example of partial differential equations with rich structure that includes mix-type and degenracy in dissipation. Hence, it provides many challenging 
mathematical problems and many of them remained unsolved after more than one hundred years from its
derivation. 

This paper aims to establish the well-posedness theory for  the
three dimensional  Prandtl equation  in  Gevrey spaces   with  the  optimal   class of  regularity
 2 that is implied by the instability results, cf. \cite{LY3D,LWY2}.  Compared with
the recent result in two space dimensions \cite{DG}, our new approach is more direct and robust to take care of the
loss of derivative in the two tangential directions.  In particular  it gives  a simpler  proof 
to  the  result in  two  dimensions  \cite{DG}. 
Hence, this  paper  is 
a complete answer to the well-posedness theory without
any structual assumption in the three dimensional setting and also shows the 
 optimality  of  the ill-posedness theories.

 Denote $\mathbb R_+^3=\big\{(x,y,z)\in\mathbb R^3; \  z>0\big\}$ and let     $(u, v)$ be   the tangential component and   $w$ be    the vertical component of the 
velocity field. Then the three dimensional  Prandtl system in  $\mathbb R_+^3$  
 reads
\begin{equation}\label{prandtl+}
\left\{
\begin{aligned}
&  \inner{\partial_t+u \partial_x   + v\partial_y +w\partial_z-\partial_{z}^2}u  + \partial_x p
=0,\quad t>0,~ (x,y,z)\in\mathbb R_+^3, \\
&  \inner{\partial_t+u\partial_x   + v\partial_y +w\partial_z -\partial_{z}^2}v+ \partial_y p
=0,\quad t>0,~ (x,y,z)\in\mathbb R_+^3, \\
&\partial_xu +\partial_yv +\partial_zw=0, \quad t>0,~ (x,y,z)\in\mathbb R_+^3, \\
&u|_{z=0} = v|_{z=0} = w|_{z=0} =0 , \quad  \lim_{z\to+\infty} (u,v) =\big(U(t,x,y), V(t,x,y)\big), \\
&u|_{t=0} =u_0, \quad v|_{t=0}=v_0, \quad   (x,y,z)\in\mathbb R_+^3,
\end{aligned}
\right.
\end{equation}
where   $(U(t,x,y), V(t,x,y))$ and $p(t,x,y)$ are the boundary traces  
of the tangential velocity field and pressure of the outer flow, 
satisfying 
\begin{eqnarray*}
\left\{
\begin{aligned}
&\partial_t U + U\partial_x U+V\partial_y U +\partial_x p=0,\\
&\partial_t V + U\partial_x V+V\partial_y V +\partial_y p=0.
\end{aligned}
\right.
\end{eqnarray*}
Here,    $p, U,V$ are given functions determined by the Euler flow. Note that   \eqref{prandtl+} is a degenerate parabolic system losing one order
derivative in the tangential variable. 
We refer to \cite{mamoudi,oleinik-3,prandtl} for the background
and  mathematical presentation of  this fundamental system.

So far, the well-posedness theories for the Prandtl equation are basically limited to the two space dimensions  except
the works by Sammartino-Caflisch \cite{Samm} in analytic function space and some recent works in Gevrey function space.  
In the   two   dimensional case, under Oleinik's monotonicity condition, there are mainly two
  analytic techniques for the well-posedness theories,  one referred to as  coordinate transformations and the  second one referred to  as cancellations.  Precisely,  the Crocco transformation was
used by Oleinik \cite{oleinik-3} for the unsteady layer to transfer the 
two dimensional  Prandtl equation into a degenerate parabolic equation. The cancellations in the convection terms were observed in recent years by two research groups independently, \cite{awxy,MW}, to overcome the difficulty of the loss of derivatives  in the system.  However, these two powerful analytic techniques are limited to the  two space dimension so far. 
 For 
three dimensions, much less is known in the well-posedness theories in Sobolev spaces. Let us also mention the work \cite{xin-zhang} on the global existence of weak solutions under an additional favorable pressure condition.

In two space dimensions, without the monotonicity condition, boundary layer separation is well expected and there are a lot of studies of the instability phenomena. Here, we only mention the works \cite{e-2} about the construction of
blowup solutions (see  also \cite{CGIM,CGM}), 
\cite{grenier} on the unstable Euler shear flow that yields instability of Prandtl equation,  \cite{GV-D, GV-N, guo} about the instability around
a   shear flow with a non-degenerate critical point,
\cite{GGN} on the instabilility even for Rayleigh's stable shear flow,
\cite{LWY2} about three space dimensional perturbation of shear flow when the initial data satisfying $U(z)\not\equiv kV(z)$ with 
a constant $k$. 
In fact, the instability result in \cite{GV-D} implies that the critical Gevrey index for well-posedness without structural
condition  is 2 and this is proved in two space dimensional  \cite{DG}. The well-posedness theories in function spaces of smooth functions was proved in \cite{Samm} with justification of the Prandtl ansatz when the data is analytic; and then it was studied in  \cite{GM} for  two space dimension with Gevrey index = $\frac 74$ that was improved in \cite{LY} to the Gevrey index  in $(1,2]$ with non-monotonic flow and then finalized in two space
dimension without any
structural condition in  \cite{DG}. In three dimensional space, we also have some work recently without  monotonicity assumption. In addition, recently, the separation singularity for stationary Prandtl system
was  studied in \cite{D-M} that justifies the Goldstein singularity.

All these results are in fact related to  the high Reynolds number limit for viscous fluid systems that is
important in both mathematics and physics. Without boundary effect, the mathematical theories
are now satisfactory (see for instance \cite{cons,mamoudi1} and references therein). 
The case with boundary is more complicated and interesting. For this,  
Kato in 1984 gave  a necessary and sufficient condition for weak convergence of viscous fluid to inviscid fluid in terms
of the vanishing energy dissipation rate in the region near the boundary. Recently, there is a series of works 
\cite{ Bardos-Titi, Bardos}  on such limit with relation to the Onsager conjecture. As for Prandtl boundary
layer, 
\cite{Mae} gave a proof when the  initial vorticity is supported away from the boundary for two dimensional  flow that was
generalized to three dimension in \cite{ZZ-2}; and recently, there are also some interesting
works on the limit to steady flow in  \cite{GN2} over a moving plate and 
in  \cite{G-S}  over a small distance, and  \cite{GM} about 
the Sobolev stability of steady shear flow in two dimensional space. For the time dependent problem, 
the stability
of Prandtl expansion in two space dimension  in Gevrey function space was studied in \cite{GMM}.

Without loss of generality we will  assume  that  $(U, V)\equiv 0.$ 
Extending our result  to the  case  of a general  outer  flow requires using some nontrivial weights similar to those in \cite{DG}.

 Then for the zero outer flow, 
 the Prandtl system \eqref{prandtl+}  can be written as   
\begin{equation}\label{prandtl}
\left\{
\begin{aligned}
&  \inner{\partial_t+u \partial_x   + v\partial_y +w\partial_z-\partial_{z}^2}u 
=0,\quad t>0,~ (x,y,z)\in\mathbb R_+^3, \\
&  \inner{\partial_t+u\partial_x   + v\partial_y +w\partial_z-\partial_{z}^2}v  =0,\quad t>0,~ (x,y,z)\in\mathbb R_+^3,  \\
&(u,v)|_{z=0} =(0,0), \quad  \lim_{z\to+\infty} (u,v) =\big(0, 0\big), \\
&(u,v)|_{t=0} =(u_0, v_0), \quad   (x,y,z)\in\mathbb R_+^3,
\end{aligned}
\right.
\end{equation}
with 
\begin{eqnarray*}
	w(t,x,y,z)=- \int_0^z\partial_x u(t,x, y,\tilde z)\,d\tilde z- \int_0^z\partial_y v(t,x, y,\tilde z)\,d\tilde z.
\end{eqnarray*}
 
Before stating our main result concerning with the well-posedness of the Prandtl system \eqref{prandtl},   we first list some notations to be used frequently throughout the paper and then introduce the   Gevrey function space.  
 
\noindent {\bf Notations}.  Throughout the paper we will  use without confusion  $\norm{\cdot}_{L^2}$ and $\inner{\cdot,\ \cdot}_{L^2}$ to denote the norm and inner product of  $L^2=L^2(\mathbb R_+^3),$  and use the notations $\norm{\cdot}_{L^2(\mathbb R_{x,y}^2)}$ and  $\inner{\cdot,\ \cdot}_{L^2(\mathbb R_{x,y}^2)}$   when the variables are specified.   Similarly for $L^\infty.$  Moreover  we also  use  $L_{x,y}^\infty(L_z^2)=L^\infty\inner{\mathbb R^2; L^2(\mathbb R_+)}$ to stands for the classical Sobolev space, so does  the Sobolev space $L_{x,y}^2(L_z^\infty).$    In the following discussion by $\partial^\alpha$ we always mean   $\partial^\alpha=\partial_x^{\alpha_1}\partial_y^{\alpha_2}$ with each multi-index  $\alpha=(\alpha_1,\alpha_2)\in\mathbb Z_+^2.$

  \begin{definition} 
\label{defgev}  Let $\ell>1/2 $ be a given number.   With each pair $(\rho,\sigma)$, $\rho>0$ and $\sigma\geq 1, $  a Banach space $X_{\rho,\sigma}$   consists of all  smooth  vector-valued functions
 $ (u, v)$  such that the Gevrey norm  $\norm{(u, v)}_{\rho,\sigma}<+\infty,$  where    $\norm{\cdot}_{\rho,\sigma}$ is defined below.   Recalling  $\partial^\alpha=\partial_x^{\alpha_1}\partial_y^{\alpha_2}$ we define 
\begin{eqnarray*}
\begin{aligned}
	 \norm{(u,v)}_{\rho,\sigma}= &\sup_{\stackrel{0\leq j\leq 5}{\abs\alpha+j\geq 7}} \frac{\rho^{\abs\alpha+j- 7}}{[\inner{\abs\alpha+j- 7}!]^{\sigma}} \Big(\big\|\comi z^{\ell+j} \partial^\alpha \partial_z^j  u\big\|_{L^2}+\big\|\comi z^{\ell+j} \partial^\alpha \partial_z^j    v\big\|_{L^2}\Big)\\
	 &+\sup_{\stackrel{0\leq j\leq 5}{\abs\alpha+j\leq 6}}   \Big(\big\|\comi z^{\ell+j} \partial^\alpha \partial_z^j  u \big\|_{L^2}+\big\|\comi z^{\ell+j} \partial^\alpha \partial_z^j  v \big\|_{L^2}\Big),
	 \end{aligned}
\end{eqnarray*}
where and throughout the paper $\comi z=(1+\abs z^2)^{1/2}.$  We call $\sigma$ the Gevrey index. 
\end{definition}

\begin{remark}
	Note that  $X_{\rho,\sigma}$ is a partial Gevrey function space.   By partial Gevrey function space,  we mean it  consists  functions that are of Gevrey class in tangential variables $x,y$ and lie in Sobolev space for normal variable $z$.       
\end{remark}

We will look for the solutions to \eqref{prandtl} in the Gevrey function space $X_{\rho,\sigma}$.   For this, the  initial data $(u_0,v_0)$  satisfy the following compatibility conditions
 \begin{equation}\label{comcon}
 \left\{
 \begin{aligned}
 &(u_0, v_0)|_{z=0}=(0,0),~~\lim_{z\rightarrow +\infty}(u_0,v_0)|=(0,0),~~(\partial_{z}^2u_0, \partial_{z}^2v_0)|_{z=0}=(0,0),\\
 &\partial_z^4u_0\big |_{z=0}=  (\partial_z u_0)\inner{\partial_x \partial_z u_0-\partial_y\partial_z v_0} |_{z=0}+2 (\partial_z v_0) \partial_y\partial_z u_0|_{z=0},\\
 &\partial_z^4v_0\big |_{z=0}= (\partial_z v_0)\inner{\partial_y\partial_z v_0-\partial_x\partial_z u_0} |_{z=0}+2(\partial_z u_0)\partial_x\partial_z u_0 |_{z=0}.	
	\end{aligned}
	\right.
	\end{equation}
 
 The main result can be stated as follows.
 \begin{theorem}
 \label{maithm1}	
 	Let $1<\sigma\leq  2.$ Suppose  the initial datum      $(u_0,v_0)$  belongs to $ X_{2\rho_0,\sigma}$ for some $\rho_0>0,$  and satisfies  the    compatibility condition   \eqref{comcon}.  
	 Then the system \eqref{prandtl} admits a unique solution $(u,v) \in L^\infty\big([0,T];~X_{\rho,\sigma}\big)$ for some $T>0$ and some $0<\rho<2\rho_0.$ 
 \end{theorem}

The rest of the paper is organized as follows.      We will prove  in Sections  \ref{secapri}-\ref{seclamdelta}   a priori estimates.    The proof of   the well-posedness for the Prandtl system is given in the last section.

\section{A priori estimate}\label{secapri}

Suppose $(u,v)\in L^\infty\inner{[0, T];~X_{\rho_0,\sigma}}$ is a solution  to the Prandtl system \eqref{prandtl} with initial datum $(u_0,v_0)\in X_{2\rho_0,\sigma}.$  This section  and Sections \ref{sec5}-\ref{seclamdelta}  are to  derive a priori estimate for  $(u,v)$.

 \subsection{Methodologies for a toy model}  Our argument is inspired by the  abstract Cauchy-Kowalewski theorem,  whose statement in general Banach scale can be found in  \cite{asa}  and the references therein; See \cite{LY,Samm}  as well for its application  to the Wellposedness theory of Prandtl equations  in analytic or Gevrey  spaces .  Let $(Y_\rho, \abs{\cdot}_\rho), 0<\rho\leq\rho_0,$ be a    Banach scale which means  $Y_\rho, 0<\rho\leq\rho_0,$ is a family of Banach spaces with norm $\abs{\cdot}_\rho$ such that $Y_{\rho_1}\subset Y_{\rho_2}$ for $0<\rho_2\leq \rho_1\leq \rho_0$. Consider  the initial value problem, with $F$ a given function, 
 \begin{eqnarray*}
 	\partial_t\vec u=F(t, \nabla \vec u),\quad\vec u|_{t=0}=\vec u_0,
 \end{eqnarray*}
 where the unknown $\vec u=\vec u(t, p)$ is an vector-valued function and $\nabla=\nabla_p.$
 Note this equation loses one order derivative and its existence theory follows from the classical Cauchy-Kowalewski theorem provided $F$ is an analytic function; meanwhile we can apply the  abstract Cauchy-Kowalewski theorem to derive  its existence in the Banach scale of {\it analytic} functions (see \cite{Samm} for instance) and the key part is to find an inequality of the following type 
 \begin{equation}\label{getm}
	   |\vec u(t)|_{ \rho}  \leq C  \int_0^{t}   \frac{ |\vec u(s)|_{\tilde \rho} }{\tilde\rho-\rho} ds+\textrm{l.o.t} +\textrm{initial data},
\end{equation} 
with $\rho<\tilde\rho$,  where here and below $\textrm{l.o.t}  $ refers to lower order terms that are easier to control and $ \textrm{initial data} $  refers to terms that are controlled by the 
 initial data. 
The intrinsic idea behind the abstract Cauchy-Kowalewski theorem is to overcome the loss of derivatives by shrinking the radius $\rho.$  More generally,  the existence theory can be extended to a Banach scale of Gevrey space   rather than of analytic space  when considering the following  
  \begin{eqnarray}\label{g2}
 	\partial_t^2 \vec u=F(t, \nabla \vec u),\quad\vec u|_{t=0}=\vec u_0, \ ~  \partial _t\vec u |_{t=0}=\vec v_0.
 \end{eqnarray}
 In fact, the above equation is equivalent to  
   \begin{equation}\label{20Jul}
   \left\{
   \begin{aligned}
   &\partial_t  \vec u=\vec v,\\
 	&\partial_t  \vec v= F(t, \nabla \vec u),\\
 	& \vec u|_{t=0}=\vec u_0,\quad \vec v|_{t=0}=\vec v_0,
 	\end{aligned}
 	\right.
 \end{equation} 
 and roughly speaking we will  lose only  1/2 rather than 1  order derivatives in each equation of \eqref{20Jul} if   $\vec v$ behaves like the $1/2$ order derivative of $\vec u.$   Then following the argument  used in analytic case,   the estimate \eqref{getm} still holds with the analytic norm therein replaced by a Gevrey norm with index $\leq 2$.  We will explain it in detail in the next paragraph. 
 
 Similar to \eqref{g2},   we replace $\partial_t$ there by a linear operator, and consider a  toy model of   Prandtl equation:
  \begin{eqnarray}\label{ef}
  \big(\partial_t+ \partial_x+ \partial_y+ \partial_z-\partial_z^2\big) ^2 \varphi   = G,  \end{eqnarray}
 with $G$ being the  linear combinations of the following types 
 \begin{eqnarray*}
    \partial^{\beta}\varphi, \quad \abs\beta\leq 1,  
 \end{eqnarray*} 
  with $\partial^\beta=\partial_x^{\beta_1}\partial_y^{\beta_2}.$
 Moreover   as explained in the previous  paragraph we rewrite \eqref{ef} as
 \begin{eqnarray}\label{eqf}
  \left\{
  \begin{aligned}
 &\big(\partial_t+ \partial_x+ \partial_y+\partial_z -\partial_z^2\big)    \varphi =\xi,\\
 &\big(\partial_t+ \partial_x+ \partial_y+\partial_z -\partial_z^2\big)    \xi =G,
 	\end{aligned}
 	\right.
 \end{eqnarray} 
 and consider $\xi$ to behave like the $1/2$ order derivative of $\varphi;$ that is,  $\xi\sim \Lambda_{x,y}^{1/2}\varphi $ where we denote by $\Lambda_{x,y}$   the Fourier multiplier of symbol $(\xi^2+\eta^2)^{1/2}$ with $(\xi,\eta)$ the dual variable of $(x,y)$.    Now let $\varphi\in Y_\rho$ with $Y_\rho$ being the Gevrey space of index 2, that is,   
 \begin{eqnarray*}
 	\norm{\partial^\alpha\varphi}_{L^2} \leq C (\abs\alpha!)^2 /\rho^{\abs\alpha}
 \end{eqnarray*}
 with $C$ independent of $\alpha\in \mathbb Z_+^2$. Then   the quantity $\frac{\rho^{\abs\alpha}}{(\abs\alpha!)^2} \norm{\partial^\alpha \varphi}_{L^2} $ is uniformly bounded with respect to $\alpha\in\mathbb Z_+^2.$    Then by interpolation inequality we have, supposing $\rho\leq 1$ without loss of generality,  
 \begin{multline}\label{adfac}
 	 \norm{\Lambda_{x,y}^{1/2}\partial^\alpha\varphi}_{L^2}\leq C\Big[\Big((\abs\alpha+1)!\Big)^2 /\rho^{\abs\alpha+1}\Big]^{1/2}\Big[(\abs\alpha!)^2 /\rho^{\abs\alpha} \Big]^{1/2}\\ \leq \tilde C \frac{\Big((\abs\alpha+1)!\Big)^2 }{\rho^{\abs\alpha+1}}\frac{1}{\abs\alpha}
 \end{multline}
 with $\tilde C$ independent of $\alpha. $     This motivates us to define 
 $\abs{\cdot}_\rho$ for each $\rho>0$  by 
 \begin{eqnarray}\label{aunor}
 |\vec b|_{\rho}=\sup_{\abs\alpha\geq 0}\frac{\rho^{\abs\alpha}}{(\abs\alpha!)^2} \norm{\partial^\alpha \varphi}_{L^2} +\sup_{\abs\alpha\geq 0}\frac{\rho^{\abs\alpha+1}}{[(\abs\alpha+1)!]^2}\abs\alpha \norm{\partial^\alpha \xi}_{L^2},
 \end{eqnarray}
 where  we use the notation $\vec b=(\varphi,\xi)$ with $\xi$ given in \eqref{eqf}  and  $\partial^\alpha=\partial_x^{\alpha_1}\partial_y^{\alpha_2}.$   Note in the above definition  there is an additional factor  $\abs\alpha$ before the norms $\norm{\partial^\alpha \xi}_{L^2} $  which follows from \eqref{adfac} since $\xi\sim\Lambda_{x,y}^{1/2}\varphi.$ 
   
Next we will derive the estimate for   the norm $|\vec b|_\rho$ defined above. Suppose $ \partial_z\varphi|_{z=0}=\xi|_{z=0}=0$.  Then
  applying  the standard energy method to \eqref{eqf} we have, for   any $\abs\alpha\geq 1$ and any $0<\rho<1$,   
 \begin{eqnarray*}
 \begin{aligned}
 &	\frac{\rho^{2\abs\alpha}}{(\abs\alpha!)^{4}} \norm{\partial^\alpha \varphi (t)}_{L^2}^2  +\frac{\rho^{2\inner{\abs\alpha+1}}}{[(\abs\alpha+1)!]^{4}} \abs\alpha^2 \norm{\partial^\alpha \xi(t)}_{L^2}^2 \\
 &\quad+\int_0^t \frac{\rho^{2\abs\alpha}}{(\abs\alpha!)^{4}} \norm{\partial_z\partial^\alpha \varphi (s)}_{L^2}^2ds+\int_0^t \frac{\rho^{2\inner{\abs\alpha+1}}}{[(\abs\alpha+1)!]^{4}} \abs\alpha^2 \norm{\partial_z\partial^\alpha \xi(s)}_{L^2}^2ds   \\
 &\qquad	\leq    2  \frac{\rho^{2 \abs\alpha}}{(\abs\alpha!)^{4}} \int_0^{t}   \norm{\partial^\alpha \xi (s)}_{L^2}\norm{\partial^{\alpha} \varphi (s)}_{L^2}  ds \\
 	&\qquad\quad+ C \abs\alpha^2 \frac{\rho^{2(\abs\alpha+1)}}{[(\abs\alpha+1)!]^{4}} \int_0^{t} \sum_{\abs\beta=1}  \norm{\partial^{\alpha+\beta} \varphi (s)}_{L^2}   \norm{\partial^\alpha \xi (s)}_{L^2} ds \\
 	&\qquad\quad+\textrm{l.o.t} +\textrm{initial data}.
 	\end{aligned}
 \end{eqnarray*}
  From the definition of $\abs{\cdot}_r$ given in \eqref{aunor},   it follows that 
 \begin{eqnarray*}
 \forall\   r>0,	 ~ \forall~ j \geq 1, \quad \norm{\partial^j \varphi }_{L^2} \leq  \frac{(j!)^2}{   r^j}|\vec b|_{  r} \  \textrm{and}\ \norm{\partial^j \xi }_{L^2} \leq \frac{1}{j} \frac{[(j+1)!]^2}{  r^{j+1}}|\vec b|_{  r}.
 \end{eqnarray*} 
We use the above estimates to compute, for any $\tilde\rho$ with $0<\rho<\tilde\rho\leq 1,$
\begin{eqnarray*}
\begin{aligned}
	&\frac{\rho^{2 \abs\alpha}}{(\abs\alpha!)^{4}} \int_0^{t}   \norm{\partial^\alpha \xi(s)}_{L^2}\norm{\partial^{\alpha} \varphi (s)}_{L^2}  ds \\
	&\leq \frac{\rho^{2 \abs\alpha}}{(\abs\alpha!)^{4}} \int_0^{t} \frac{1}{\abs\alpha} \frac{[(\abs\alpha+1)!]^2}{{\tilde \rho}^{\abs\alpha+1}}    \frac{(\abs\alpha !)^2}{{\tilde \rho}^{\abs\alpha}}|\vec b(s)|_{  \tilde\rho}^2  ds\\
	&\leq    4 \int_0^{t}   \frac{\abs\alpha }{\tilde \rho}  \frac{\rho^{2 \abs\alpha}}{\tilde \rho^{2\abs\alpha}}    |\vec b(s)|_{  \tilde\rho}^2  ds \leq    4 \int_0^{t}   \frac{\abs\alpha }{\tilde \rho}  \Big(\frac{\rho}{\tilde \rho}\Big)^{ \abs\alpha}   |\vec b(s)|_{  \tilde\rho}^2  ds\leq 4\int_0^{t}   \frac{   |\vec b(s)|_{  \tilde\rho}^2}{\tilde \rho-\rho} ds,
	\end{aligned}
\end{eqnarray*}
the last inequality using the fact  that    
 for any integer $k\geq 1$ and for any  pair $(\rho,\tilde \rho)$ with $0<\rho<\tilde\rho\leq 1,$  
\begin{equation}
\label{factor}
   k\inner{\frac{\rho}{\tilde\rho}}^k\leq 	\frac{k}{\tilde\rho} \inner{\frac{\rho}{\tilde\rho}}^k\leq\frac{1}{\tilde\rho-\rho}.
\end{equation}
Note that  the first inequality in \eqref{factor} is obvious since $\tilde \rho\leq 1$ and the second one follows from the fact that
\begin{eqnarray*}
	\frac{1}{1-\frac{\rho}{\tilde\rho}}=\sum_{j=0}^\infty \Big(\frac{\rho}{\tilde\rho}\Big)^j\geq  k \Big(\frac{\rho}{\tilde\rho}\Big)^k. 
\end{eqnarray*}
Applying the similar argument as above gives also
\begin{multline*}
	\abs\alpha^2 \frac{\rho^{2(\abs\alpha+1)}}{[(\abs\alpha+1)!]^{4}} \int_0^{t} \sum_{\abs\beta=1}  \norm{\partial^{\alpha+\beta} \varphi (s)}_{L^2}  \norm{\partial^\alpha \xi (s)}_{L^2} ds\\
	 \leq 2 \int_0^{t}    \abs\alpha \frac{\rho^{2( \abs\alpha+1)}}{\tilde \rho^{2(\abs\alpha+1)}}       |\vec b(s)|_{  \tilde\rho}^2  ds\leq 2 \int_0^{t}   \frac{     |\vec b(s)|_{  \tilde\rho}^2}{\tilde \rho-\rho} ds.
\end{multline*}
 Then
combining the above inequalities we conclude
\begin{eqnarray*}
	   |\vec b(t)|_{ \rho}^2 \leq C  \int_0^{t}   \frac{ |\vec b(s)|_{\tilde \rho}^2 }{\tilde\rho-\rho} ds+\textrm{l.o.t} +\textrm{initial data}.
\end{eqnarray*}
  This estimate enables us to follow the argument for proving abstract Cauchy-Kowalewski theorem,   seeing  for instance \cite[Section 8]{LY} and Section \ref{sec8}  below for the detailed discussion,   to obtain the existence of solution to \eqref{eqf}.  
  
  Finally we remark that in the above argument we do not use the diffusion property of the system \eqref{eqf} in normal variable $z$     when dealing with the tangential derivatives of $\varphi$ and $\xi$.     
 
\subsection{Auxilliary functions  and statement of  a priori estimate}\label{subaux}
Inspired by Dietert and G\'erard-Varet's work   \cite{DG}  let 
 $\mathcal U$ be  a solution to the linear initial-boundary problem
 		\begin{eqnarray}\label{mau}
 		\left\{
 		\begin{aligned}
&   \big(\partial_t+u\partial_x+v\partial_y+w\partial_z-\partial_z^2\big)    \int_0^z\mathcal U(t,x,y,\tilde z) d\tilde z  =  -\partial_x w(t,x,y,z),\\
& \mathcal U|_{t=0}=0, \quad \partial_z\mathcal U|_{z=0}=\mathcal U|_{z\rightarrow+\infty}=0.
   \end{aligned}
   \right.
	\end{eqnarray} 
The existence of  $\mathcal U$  just follows from the standard parabolic theory.  In fact we first construct a solution $f$ to the following
\begin{eqnarray} \label{eqoff}
 		\left\{
 		\begin{aligned}
&    \big(\partial_t+u\partial_x+v\partial_y+w\partial_z-\partial_z^2\big)     f  =-\partial_x w \\
& f|_{t=0}=0, \quad f|_{z=0}=\partial_zf|_{z\rightarrow+\infty}=0,
   \end{aligned}
   \right.
	\end{eqnarray} 
	and then define $\mathcal U=\partial_z f$ which will solve \eqref{mau}.   Similarly let  $\widetilde{\mathcal U}$ solve
\begin{eqnarray}\label{mav}
 		\left\{
 		\begin{aligned}
&    \big(\partial_t+u\partial_x+v\partial_y+w\partial_z-\partial_z^2\big)    \int_0^z\widetilde{\mathcal U}(t,x,y,\tilde z) d\tilde z  =  -\partial_y w(t,x,y,z),\\
& \widetilde{\mathcal U}|_{t=0}=0, \quad \partial_z\widetilde{\mathcal U}|_{z=0}=\widetilde{\mathcal U}|_{z\rightarrow+\infty}=0.
   \end{aligned}
   \right.
	\end{eqnarray} 
Moreover  we define    $\lambda, \delta$ and $\tilde\lambda, \tilde \delta$   as follows 
\begin{eqnarray}\label{laga}
\left\{
\begin{aligned}
	&\lambda= \partial_x u-(\partial_z u)\int_0^z\mathcal U  d\tilde z,\qquad \tilde \lambda= \partial_y u-(\partial_z u)\int_0^z\widetilde{\mathcal U}  d\tilde z,\\
	&\delta= \partial_x v-(\partial_z v)\int_0^z\mathcal U  d\tilde z,\qquad \tilde\delta= \partial_y v-(\partial_z v)\int_0^z\widetilde{\mathcal U} d\tilde z, 
	\end{aligned}
\right.
\end{eqnarray}
that are to be used to  derive the estimate on $\mathcal U$ and $\widetilde{\mathcal U}.$  As to be seen later    a new type of cancellation will be applied 
when deriving the equation for $\lambda$ and this enables us to eliminate the bad term involving $\partial_x w$ that loses one order derivative.   The idea of observing cancellation mechanism to overcome the lost
of derivative was initiated independently by  Alexandre-Wang-Xu-Yang \cite{awxy} and Masmoudi-Wong \cite{MW}, where they considered the two-dimensional case and  introduced the good-unknown of the type
$\partial_x^m\partial_zu-\big(\partial_z^2u/\partial_zu\big) \partial_x^mu$ for $m\geq 1$,  under  the Oleinik's monotonicity  condition $\partial_zu \neq  0$; see also \cite{GM, LY} for other type of  cancellations   when exploiting the well-posedness theory for Prandtl equation without analyticity or monotonicity.  Note that we can not  apply directly the above good-unknown  in our case since $\partial_z u$ may vanish,  and the novelty here  is the introduction of the auxilliary functions $\lambda,\delta,\cdots,$  in \eqref{laga},   which are the generalized case of the good-unknown aforementioned.  Our argument will combine  a  new cancellation for  these auxilliary functions with the idea of introducing $\mathcal U$ initiated by  Dietert and G\'erard-Varet \cite{DG}.  In fact these auxilliary functions play an import role when performing the energy estimate for $\mathcal U.$  Precisely,  if we    apply $\partial_z$  to \eqref{mau}  then we have an evolution equation for $\mathcal U$  with  $\lambda$ and $\delta$   as   source terms that lead to the loss of one order derivatives.  Our  observation is  that we only lose half rather than one order derivative in the equation for $\mathcal U$  since we have additionally evolution equations for $\lambda$ and $\delta$ that do not lose derivatives anymore.   This  enables  us  to  close the energy  estimates for $\mathcal U$ in the Gevrey space of index up to $ 2$ rather than in analytic space.   We will explain in more detail in the next paragraph.   By virtue of these functions in \eqref{laga}  we can apply the idea in the previous subsection to derive the desired estimate for $\mathcal U$;  this is essentially different from the treatment presented  in Dietert and G\'erard-Varet's work  \cite{DG}.    

Next  we will  explain the main difficulties and the new ideas   introduced in this
paper.  We first estimate $u$.  
Applying  $\partial_x$  to the first equation in \eqref{prandtl} yields
\begin{equation}\label{zu+}
	\big(\partial_t + u\partial_x    +v\partial_y	 +w\partial_z   -\partial _{z}^2\big)\partial_xu=-(\partial_x w)\partial_z u-(\partial_x u)\partial_xu-(\partial_x v)\partial_yu.
\end{equation} 
Note we  lose one order tangential derivatives in  $\partial_x w$ which is the main difficulty for the existence theory of Prandtl equation.  To overcome the loss of derivatives  we introduce  a new  cancellation in the system.     Multiplying the equation \eqref{mau} by  $\partial_z u $    and then subtracting the resulting equation by \eqref{zu+};  this eliminates the  term   $(\partial_x w)\partial_zu$ that loses derivatives  and yields 
\begin{multline}\label{edif}
	  \big(\partial_t + u\partial_x    +v\partial_y	 +w\partial_z   -\partial _{z}^2\big) \overbrace{\big[\partial_x  u-(\partial_z  u)\int_0^z  \mathcal U d\tilde z\big]}^{=\lambda}\\ 
 	 =-(\partial_x u)\partial_xu- (\partial_xv)\partial_yu - \big[(\partial_y v)\partial_zu- ( \partial_yu)\partial_zv\big]  \int_0^z\mathcal U d\tilde z+2(\partial_z^2 u)\mathcal U.
\end{multline}
Note the above equation for $\lambda$ doesn't lose derivatives if considering  $\lambda$ has the same order as that of $\partial_x u$ and $\mathcal U$.   Thus we can derive the estimate for $\lambda$  from  the equation \eqref{edif} above  and  as a result the estimate on $\partial_x u$ will follow provided  we can  control $(\partial_z  u)\int_0^z  \mathcal U d\tilde z.$

To control  $(\partial_z  u)\int_0^z  \mathcal U d\tilde z$  we can not perform the energy estimate  from its equation   \eqref{mau}  since we lose one order derivatives  caused by the source term $\partial_xw.$   Instead we will control $(\partial_z  u)\int_0^z  \mathcal U d\tilde z$   in terms of $\mathcal U$ which solves the following equation,    applying  $\partial_z$ to \eqref{mau},
\begin{multline}\label{ueqla}
\big(\partial_t+u\partial_x+v\partial_y+w\partial_z-\partial_z^2\big) \mathcal U  \\
=\underbrace{\partial_x^2 u-\partial_x\Big[(\partial_z u)\int_0^z\mathcal U  d\tilde z\Big]}_{= \partial_x\lambda}+\underbrace{\partial_x\partial_yv-\partial_y\Big[(\partial_z v)\int_0^z \mathcal U  d\tilde z\Big]}_{=\partial_y\delta}+\textrm{l.o.t.}, 	
\end{multline}
recalling $\lambda$ and $\delta$ are given by \eqref{laga}.   Furthermore  it follows from \eqref{edif}  that
\begin{eqnarray*} 
\begin{aligned}
&	\big(\partial_t+u\partial_x+v\partial_y+w\partial_z-\partial_z^2 \big)  \partial_x \lambda\\
	& \quad= - \Big[(\partial_x  u)\partial_x \lambda+ (\partial_x v) \partial_y\lambda+ (\partial_x w) \partial_z\lambda\Big]\\
	&\qquad\ \ - \partial_x \Big[(\partial_xu)\partial_xu+(\partial_xv)\partial_yu + \big[(\partial_y v)\partial_zu- ( \partial_yu)\partial_zv\big]  \int_0^z\mathcal U d\tilde z-2(\partial_z^2 u)\mathcal U\Big].
	\end{aligned}
\end{eqnarray*}
Similarly for $\partial_y \delta$.  Then combining the two equations above gives
\begin{multline}\label{macu}
\big(\partial_t+u\partial_x+v\partial_y+w\partial_z-\partial_z^2\big)^2 \mathcal U  \\
=\textrm{terms involving the second order derivatives}+\textrm{l.o.t.}, 
\end{multline}
Then the situation is similar to  that for the model equation \eqref{ef} or \eqref{eqf},  with $\varphi$ and $\xi$ therein corresponding  to  $\mathcal U$ and $\partial_x \lambda+\partial_y\delta$ respectively. 
    Inspired  by  the treatment of  the model equation \eqref{ef} or \eqref{eqf} and  the definition in \eqref{aunor},   it is natural to consider the uniform upper bound  with respect to $\alpha$  of the following   norm:
\begin{eqnarray*}
	 \frac{\rho^{\abs\alpha-6}}{ [( \abs\alpha-6)!]^{\sigma}}  \big\| \partial^{\alpha} \mathcal U \big\|_{L^2}  	 + \frac{ \rho^{(\abs\alpha+1)-6}}{ [\inner{\abs\alpha+1-6}!]^{\sigma}}  (\abs\alpha+1) \Big( \norm{ \partial^\alpha\partial_x\lambda}_{L^2}+ \norm{\partial^\alpha\partial_y\delta}_{L^2} \Big)
\end{eqnarray*}  
or its equivalence 
 \begin{eqnarray*}
	 \frac{\rho^{\abs\alpha-6}}{ [( \abs\alpha-6)!]^{\sigma}}  \big\| \partial^{\alpha} \mathcal U \big\|_{L^2}  	 + \frac{ \rho^{\abs\alpha-6}}{ [\inner{\abs\alpha-6}!]^{\sigma}}  \abs\alpha \Big( \norm{ \partial^\alpha\lambda}_{L^2}+ \norm{\partial^\alpha \delta}_{L^2} \Big) 
	 \end{eqnarray*}     
with $\sigma\leq 2,$ recalling $\mathcal U$ has the same order as that of $\partial_x u.$   Precisely, we will define $\abs{\cdot}_{\rho,\sigma}$ as below,   similar to the definition in \eqref{aunor}. 
 We use the notation
 \begin{eqnarray*}
  \vec a=(u, v,\mathcal U, \widetilde{\mathcal U}, \lambda, \tilde \lambda, \delta, \tilde\delta) 	
 \end{eqnarray*}
  Recall $\mathcal U, \widetilde{\mathcal U}$ are  given by \eqref{mau} and \eqref{mav}, and $  \lambda, \tilde \lambda, \delta, \tilde\delta$ are defined by \eqref{laga}. 
  
    \begin{definition}  
\label{gevspace}	
Let $\norm{(u,v)}_{\rho,\sigma}$ be given in Definition \ref{defgev}. With the notation $\vec a$ above,  we define $ \abs{\vec a}_{\rho,\sigma}$ by setting 
\begin{eqnarray*}
\begin{aligned}
	  \abs{\vec a}_{\rho,\sigma}  =& \norm{(u,v)}_{\rho,\sigma}\\
	  &+\sup_{\abs\alpha\geq 6} \frac{\rho^{\abs\alpha-6}}{ [( \abs\alpha-6)!]^{\sigma}} \Big(\big\| \partial^{\alpha} \mathcal U \big\|_{L^2}+\big\|\partial^\alpha\widetilde{ \mathcal U}\big\|_{L^2}  \Big)\\
	  &+\sup_{\abs\alpha\leq 5}   \Big(\big\| \partial^\alpha \mathcal U \big\|_{L^2}+\big\|\partial^\alpha \widetilde{ \mathcal U}\big\|_{L^2}  \Big) \\
	  &+\sup_{\abs\alpha\geq 6}\frac{ \rho^{\abs\alpha-6}}{ [\inner{\abs\alpha-6}!]^{\sigma}}  \abs\alpha \Big( \norm{ \partial^\alpha\lambda}_{L^2}+ \norm{\partial^\alpha\delta}_{L^2}+\norm{ \partial^\alpha\tilde\lambda}_{L^2}+ \norm{\partial^\alpha\tilde\delta}_{L^2}\Big)
\\
&
	+\sup_{\abs\alpha\leq 5} \abs\alpha \Big( \norm{ \partial^\alpha\lambda}_{L^2}+ \norm{\partial^\alpha\delta}_{L^2}+\norm{ \partial^\alpha\tilde\lambda}_{L^2}+ \norm{\partial^\alpha\tilde\delta}_{L^2}\Big).
\end{aligned}
\end{eqnarray*}
Note there is an additional factor  $\abs\alpha$ before the $L^2$-norms of $\partial^\alpha \lambda,  \partial^\alpha\delta$ and $ \partial^\alpha\tilde\lambda, \partial^\alpha\tilde\delta$.
 \end{definition}
 
 \begin{remark}
 The auxilliary functions $\mathcal U, \lambda,\delta$ are introduced for treating the derivatives $\partial_x^m$ and meanwhile $\widetilde{ \mathcal U}, \tilde\lambda, \tilde\delta$   are for  $\partial_y^m$. Then the estimate for the general $\partial^\alpha=\partial_x^{\alpha_1}\partial_y^{\alpha_2}$ will follow as well using  the  relationship
\begin{equation}\label{realp}
 	\forall~\alpha\in\mathbb Z_+^2,~\forall~ F\in H^\infty,\quad \norm{\partial^\alpha F}_{L^2}\leq \norm{\partial_x^{\abs\alpha}F}_{L^2 } + \norm{\partial_y^{\abs\alpha} F}_{L^2}.
 \end{equation}
 In this paper we will focus on  performing only  the estimates for  $\partial_x^m$, since the estimates for $\partial_y^m$ can be treated symmetrically in the same way.
 \end{remark}

Now we are ready to state the main  a priori estimate.    We will present in detail the proof of Theorem \ref{maithm1} for $\sigma\in [3/2,2]$.  Note that the constraint  $\sigma\geq 3/2$ is not essential and indeed it is just a technical assumption  for clear presentation. 
 We refer to \cite[Section 8]{LY3D}  for the explanation  how to modify the proof for the case when $1<\sigma<3/2$.  We make the following  low regularity assumption 
 that will be checked in the last section of the paper: 
 
\begin{assumption}\label{assmain}
	Let  $X_{\rho,\sigma}$ be the Gevrey function space  equipped with the norm $\norm{\cdot}_{\rho,\sigma}$ given in  Definition \ref{defgev}.    Suppose   $(u,v)\in L^\infty\inner{[0, T];~X_{\rho_0,\sigma}}$ for some  $0<\rho_0\leq 1$ and $\sigma\in[3/2,  2]$  is a solution to  the Prandtl system  \eqref{prandtl} with initial datum $(u_0,v_0)\in X_{2\rho_0,\sigma}$.   Without loss of generality we may assume $T\leq 1$.  Moreover  we  suppose  that there exists a constant $C_*$     such that for any $ t\in[0,T],$ 
   \begin{equation}\label{condi1}
   \sup_{\stackrel{0\leq j\leq 5}{\abs\alpha+j\leq 10}}  \Big(\big\|\comi z^{\ell+j} \partial^\alpha \partial_z^j   u(t)\big\|_{L^2}+\big\|\comi z^{\ell+j} \partial^\alpha \partial_z^j  v(t)\big\|_{L^2}\Big) \leq  C_*,
\end{equation}
where the constant $C_*\geq 1$  depends only on $\norm{(u_0,v_0)}_{2\rho_0, \sigma}$,  the Sobolev embedding  constants and the numbers $\rho_0, \sigma, \ell$ that are given in Definition \ref{defgev}. 
\end{assumption}

\begin{theorem}[A priori estimate in Gevrey space]\label{apriori}
Under Assumption \ref{assmain} above, we can find   
     two constant   $C_1, C_2\geq 1,$     such that    the estimate 
 \begin{multline*} 
 	\abs{\vec a (t)}_{\rho,\sigma}^2\leq C_1 \norm{(u_0, v_0)}_{2\rho_0, \sigma}^2  
 	+ e^{C_2C_*^2} \int_{0}^{t} \inner{\abs{\vec a(s)}_{\rho,\sigma}^2+\abs{\vec a(s)}_{\rho,\sigma}^4} \,ds \\
 	+   e^{C_2C_*^2}\int_{0}^{t}\frac{ \abs{\vec a(s)}_{\tilde\rho,\sigma}^2}{\tilde \rho-\rho}\,ds
\end{multline*}	 
 holds for any pair $(\rho,\tilde\rho)$ with   $0<\rho<\tilde \rho<\rho_0$ and any $t\in[0,T],$   where the constant $C_1$ can be computed explicitly and  the constant $C_2$ depends only on the Sobolev embedding  constants and the numbers $\rho_0, \sigma, \ell$  given in Definition \ref{defgev}. Both   $C_1$ and $C_2$  are independent of the constant $C_*$ given in \eqref{condi1}. 
 \end{theorem}

\section{Estimate on   $\partial^\alpha\mathcal U $ and $\partial^\alpha\widetilde{ \mathcal U} $}
  \label{sec5}
  
  To prove the a  priori estimate  stated in Theorem \ref{apriori},  we will proceed through this section and Sections \ref{secofuv}-\ref{seclamdelta} to derive the upper bound for the terms involved in  Definition \ref{gevspace}  of $\abs{\vec{a}}_{\rho,\sigma}$. 
For the argument presented in  Sections \ref{sec5}-\ref{seclamdelta} we always suppose Assumption \ref{assmain} is fulfilled by  $(u,v)\in L^\infty\inner{[0, T];~X_{\rho_0,\sigma}}$.

To simplify the notation,  we use from now on    the two capital letters $C_1,  C$ to denote some generic  constant that may vary from line to line,  both
depending only  on  the Sobolev embedding  constants and the numbers $\rho_0, \sigma, \ell$   given in Definition \ref{defgev}  but   {\it  independent of} the constant $C_*$ in \eqref{condi1} and the order of derivatives denoted by $m$.  

In this part we will derive the upper bound for the terms involving  $\mathcal U$ and $ \widetilde {\mathcal U}$  in Definition \ref{gevspace} of  $\abs{\vec a}_{\rho,\sigma}$.   Recall   $\mathcal U$ and $  \widetilde {\mathcal U}$   solve respectively the equations  \eqref{mau} \and \eqref{mav}. 
 	
		\begin{proposition}\label{prolambda} Under Assumption \ref{assmain} we have, 
 	  for any  $t\in[0,T]$ and   for  any pair $\inner{\rho,\tilde\rho}$ with  $0<\rho<\tilde\rho< \rho_0\leq 1$,  
 	  \begin{multline*}
		 	\sup_{\abs\alpha\geq 6}\frac{\rho^{2(\abs\alpha-6)}}{   [\inner{\abs\alpha-6}!]^{2\sigma}}   \norm{\partial^{\alpha} \mathcal U(t)}_{L^2}^2 +\sup_{\abs\alpha\leq 5} \norm{\partial^{\alpha} \mathcal U(t)}_{L^2}^2\\ \leq  CC_*\bigg(        \int_0^{t}   \big( \abs{\vec a(s)}_{ \rho,\sigma}^2+\abs{\vec a(s)}_{ \rho,\sigma}^4\big)   \,ds+ \int_0^{t}  \frac{  \abs{\vec a(s)}_{ \tilde\rho,\sigma}^2}{\tilde\rho-\rho}\,ds\bigg),
\end{multline*}
where $C_*\geq 1$ is the constant given in \eqref{condi1}.  
Symmetrically, the same upper bound also holds with $ \mathcal U$ replaced by $ \widetilde{ \mathcal U}$.   
\end{proposition}

	 We first derive the evolution equation for $\partial_x^m\mathcal U$.  Applying $\partial_z$ to \eqref{mau} yields
\begin{multline}\label{eqfou}
	 \big(\partial_t+u\partial_x+v\partial_y+w\partial_z-\partial_z^2\big) \mathcal U\\
	 =\partial_x^2u+\partial_y\partial_x v-(\partial_zu)\partial_x \int_0^z\mathcal U d\tilde z -(\partial_zv)\partial_y \int_0^z\mathcal U d\tilde z +(\partial_xu+\partial_y v)\mathcal U, 
\end{multline}	
and thus,    using  the representation of $\lambda$ and $\delta$ given in \eqref{laga},
\begin{multline*}
	 \big(\partial_t+u\partial_x+v\partial_y+w\partial_z-\partial_z^2\big) \mathcal U\\	=\partial_x\lambda+\partial_y\delta+(\partial_x\partial_zu+\partial_y\partial_zv)\int_0^z\mathcal U d\tilde z  +(\partial_xu+\partial_y v)\mathcal U, 
\end{multline*}	
	Then, applying $\partial_x^{m}$ to the above equation we get
	\begin{multline*}
	 \big(\partial_t+u\partial_x+v\partial_y+w\partial_z-\partial_z^2\big) \partial_x^{m} \mathcal U\\
	 =-\sum_{j=1}^{m}{m\choose j}\Big[(\partial_x^j u) \partial_{x}^{m-j+1}\mathcal U+(\partial_x^j v)\partial_{x}^{m-j}\partial_y\mathcal U +(\partial_x^j w)\partial_{x}^{m-j}\partial_z\mathcal U\Big] 
	\\  + \partial_x^{m}\big(\partial_x\lambda+\partial_y\delta\big)+\partial_x^{m}\Big[(\partial_x\partial_zu+\partial_y\partial_zv)\int_0^z\mathcal U d\tilde z  +(\partial_xu+\partial_y v)\mathcal U\Big]. 
\end{multline*} 
Taking the scalar product with $\partial_x^m\mathcal U$ and observing $\mathcal U|_{t=0}=\partial_z\mathcal U|_{z=0}=0$ gives
\begin{equation}\label{uma+}
\begin{aligned}
	&\frac{1}{2}\norm{\partial_x^{m} \mathcal U(t)}_{L^2}^2 +\int_0^{t} \norm{\partial_z\partial_x^{m} \mathcal U(s)}_{L^2}^2ds\\
	&  =\int_0^{t} \Big(\big(\partial_s+u\partial_x+v\partial_y+w\partial_z-\partial_z^2\big) \partial_x^{m} \mathcal U ,\ \partial_x^{m} \mathcal U \Big)_{L^2}ds\\
	& =-\int_0^{t} \Big( \sum_{j=1}^{m}{m\choose j}\Big[(\partial_x^j u) \partial_{x}^{m-j+1}\mathcal U+(\partial_x^j v)\partial_{x}^{m-j}\partial_y\mathcal U  \Big],\ \partial_x^{m} \mathcal U \Big)_{L^2}ds\\
	& ~ \quad -\int_0^{t} \Big( \sum_{j=1}^{m}{m\choose j} (\partial_x^j w)\partial_{x}^{m-j}\partial_z\mathcal U ,  \partial_x^{m} \mathcal U \Big)_{L^2}ds+\int_0^{t} \big( \partial_x^{m} (\partial_x\lambda+\partial_y\delta),  \partial_x^{m} \mathcal U\big)_{L^2}ds\\
	&\quad\quad +\int_0^{t} \Big(  \partial_x^{m}\Big[(\partial_x\partial_zu+\partial_y\partial_zv)\int_0^z\mathcal U d\tilde z  +(\partial_xu+\partial_y v)\mathcal U\Big],\ \partial_x^{m} \mathcal U\Big)_{L^2}ds.
	\end{aligned}
\end{equation}
Next we derive the upper bound for the terms on the right-hand  side through the following three lemmas. 

\begin{lemma}\label{lemhig}
	Under the same assumptions as in Proposition \ref{prolambda} we have,  for any $m\geq 6,$  any $t\in[0,T]$ and   for  any pair $\inner{\rho,\tilde\rho}$ with  $0<\rho<\tilde\rho< \rho_0\leq 1$,
	\begin{eqnarray*}
		\int_0^{t} \Big( \partial_x^{m}\big(\partial_x\lambda+\partial_y\delta\big),\ \partial_x^{m} \mathcal U \Big)_{L^2}dt \leq       C\frac{ [(m-6)!]^{ 2\sigma}  }{  \rho^{ 2(m-6)}}    \int_0^{t}  \frac{\abs{\vec a(s)}_{\tilde\rho,\sigma} ^2 }{ \tilde\rho-\rho}       ds. 
	\end{eqnarray*}
\end{lemma}

\begin{proof}
It follows from Definition \ref{gevspace} of $\abs{\vec a}_{r,\sigma}$ that,  for any $\alpha\in \mathbb Z_+^2$  and  for any $r>0$, 
\begin{equation}\label{elam}
\abs\alpha\Big(	\norm{\partial^\alpha\lambda}_{L^2}+\norm{ \partial^\alpha\delta}_{L^2}\Big)\leq 
	\left\{
	\begin{aligned}
	& \frac{   [\inner{\abs\alpha-6}!]^{ \sigma}}{r^{ (\abs\alpha-6)}}\abs{\vec a}_{r,\sigma},\quad {\rm if}~\abs\alpha \geq 6,\\
	& \abs{\vec a}_{r,\sigma},  \quad {\rm if}~ \abs\alpha \leq 5,
	\end{aligned}
	\right.
\end{equation}
  and that, observing $\ell>1/2,$ 
 \begin{multline}\label{uint+}
 	  \big\| \comi z^{-\ell-{1\over 2}}  \int_0^z    \partial^\alpha  \mathcal U d\tilde z\big\|_{L^2}+	\big\|  \comi z^{-1/2} \int_0^z     \partial^\alpha \mathcal U d\tilde z\big\|_{L_{x,y}^2(L_z^\infty)}\\
 	  \leq C\norm{\partial^\alpha  \mathcal U}_{L^2}\leq 
 	 \left\{
 	 \begin{aligned} 
 	& C\frac{[(\abs\alpha-6)!]^\sigma}{r^{\abs\alpha-6}}\abs{\vec a}_{r,\sigma},\quad {\rm if}~\abs\alpha \geq 6,\\
 	 &C	\abs{\vec a}_{r,\sigma}, \quad {\rm if}~\abs\alpha \leq 5.
 	 \end{aligned}
\right.
 	  \end{multline}
Using the above   estimates we compute 
\begin{multline*} 
	\int_0^{t} \Big( \partial_x^{m}\big(\partial_x\lambda+\partial_y\delta\big),\ \partial_x^{m} \mathcal U \Big)_{L^2}dt \leq     \int_0^{t} \frac{1}{m+1}\frac{[(m-5)!]^{ \sigma}  }{ \tilde\rho^{ m-5}}      \frac{[(m-6)!]^{ \sigma}  }{ \tilde\rho^{ m-6}}   \abs{\vec a(s)}_{\tilde\rho,\sigma} ^2 ds\\
	 \leq   C  \int_0^{t}  \frac{m^{\sigma-1} }{ \tilde\rho}      \frac{[(m-6)!]^{ 2\sigma}  }{ \tilde\rho^{ 2(m-6)}}   \abs{\vec a(s)}_{\tilde\rho,\sigma} ^2ds \leq \frac{C[(m-6)!]^{ 2\sigma}  }{  \rho^{ 2(m-6)}}    \int_0^{t}  \frac{\abs{\vec a(s)}_{\tilde\rho,\sigma} ^2 }{ \tilde\rho-\rho}       ds, 
\end{multline*}
the last inequality holding because $\sigma\leq 2$ and  
\begin{eqnarray*}
	\frac{m}{\tilde \rho} \frac{1}{\tilde\rho^{2(m-6)}}  =  \frac{1}{\rho^{2(m-6)}}\frac{m}{\tilde\rho}\frac{\rho^{2(m-6)}}{\tilde\rho^{2(m-6)}}\leq C \frac{1}{\rho^{2(m-6)}} \frac{m-6}{\tilde\rho}\Big(\frac{\rho}{\tilde\rho}\Big)^{m-6}    \leq  C\frac{1}{\rho^{2(m-6)}}  \frac{1}{\tilde\rho-\rho}
\end{eqnarray*}
due to \eqref{factor}. 
The proof of Lemma \ref{lemhig} is completed.   
\end{proof}

 	 \begin{lemma}\label{lemmac}
 	 	Under the same assumption as in Proposition \ref{prolambda} we have,  for any $m\geq 6,$  any $t\in[0,T]$ and   for  any pair $\inner{\rho,\tilde\rho}$ with  $0<\rho<\tilde\rho< \rho_0\leq 1$,
   \begin{eqnarray*}
   \begin{aligned}
   &	-\int_0^{t} \Big( \sum_{j=1}^{m}{m\choose j}\Big[(\partial_x^j u) \partial_{x}^{m-j+1}\mathcal U+(\partial_x^j v)\partial_{x}^{m-j}\partial_y\mathcal U  \Big],\ \partial_x^{m} \mathcal U \Big)_{L^2}ds\\
  & \qquad	-\int_0^{t} \Big( \sum_{j=1}^{m}{m\choose j} (\partial_x^j w)\partial_{x}^{m-j}\partial_z\mathcal U,\ \partial_x^{m} \mathcal U \Big)_{L^2}ds\\
   &	\leq {1\over 2} \int_0^{t} \norm{\partial_z\partial_x^{m} \mathcal U }_{L^2}^2 ds+ C \frac{[\inner{m-6}!]^{2\sigma}}{\rho^{2(m-6)}} \bigg( \int_0^{t}  \big(\abs{\vec a(s)}_{ \rho,\sigma}^3+   \abs{\vec a(s)}_{ \rho,\sigma}^4 \big)ds \bigg)\\
    &	\qquad +CC_*  \frac{[\inner{m-6}!]^{2\sigma}}{\rho^{2(m-6)}}      \int_0^{t}   \frac{\abs{\vec a(s)}_{\tilde \rho,\sigma}^2}{\tilde\rho-\rho} ds ,
   	\end{aligned}
   \end{eqnarray*}
where $C_*$ is the constant in \eqref{condi1}.
 	 \end{lemma}
 	 
 	 \begin{proof}
 	 We  treat the first term on the left side and write
 	 \begin{multline}\label{feu}
	\sum_{j=1}^{m}{{m}\choose j}  \norm{ (\partial_x^j u) \partial_x^{m-j+1}\mathcal U}_{L^2}\\
	\leq 	\sum_{j=1}^{[m/2]}{{m}\choose j}  \norm{ \partial_x^j u }_{L^\infty} \norm{  \partial_x^{m-j+1}\mathcal U}_{L^2} \\
	+ \sum_{j= [m/2]+1}^m {{m}\choose j}  \norm{ \partial_x^j u}_{L_{x,y}^2(L_z^\infty)}\norm{ \partial_x^{m-j+1} \mathcal U}_{L_{x,y}^\infty(L_z^2)},
\end{multline}
where as standard, $[p] $ denotes the largest integer less than or equal to $p.$   We need 
the following Sobolev embedding  inequalities:	
\begin{equation}
	\label{soblev}
	\left\{
	\begin{aligned}
	&\norm{F}_{L^\infty(\mathbb R_{x,y}^2)}\leq  \sqrt 2\Big(\norm{F}_{L_{x,y}^2}+\norm{\partial_x F}_{{L_{x,y}^2}}+\norm{\partial_y F}_{{L_{x,y}^2}}+\norm{\partial_x \partial_y F}_{{L_{x,y}^2}}\Big),\\
	&\norm{F}_{L^\infty} \leq   2\Big(\norm{ F}_{L^2}+\norm{\partial_{x}  F}_{L^2}+\norm{\partial_{y} F}_{L^2}+\norm{  \partial_{z}F}_{L^2}\Big)\\ &\qquad \qquad+2\Big(\norm{\partial_{x} \partial_{y}F}_{L^2}+\norm{\partial_{x} \partial_{z}F}_{L^2}+\norm{\partial_{y} \partial_{z}F}_{L^2}+\norm{\partial_{x}\partial_y \partial_{z}F}_{L^2}\Big),
	\end{aligned}
	\right. 
	\end{equation}
and moreover   it follows that  the definition of $\abs{\vec a}_{r,\sigma}$ and Assumption \ref{assmain} that,  for any $\alpha\in \mathbb Z_+^2,$  any $0\leq j\leq 5$  and   any $r>0$,   
\begin{multline}\label{emix}
	\norm{\comi z^{\ell+j}\partial^\alpha\partial_z^j u}_{L^2}+\norm{\comi z^{\ell+j}\partial^\alpha\partial_z^j v}_{L^2}\\
	\leq 
	\left\{
	\begin{aligned}
	& \frac{   [\inner{\abs\alpha+j-7}!]^{ \sigma}}{r^{ (\abs\alpha+j-7)}}\abs{\vec a}_{r,\sigma},\  {\rm if}~\abs\alpha+j\geq 7,\\
	& \min\Big\{\abs{\vec a}_{r,\sigma},  \  C_*\Big\},\quad {\rm if}~ \abs\alpha+j\leq 6,
	\end{aligned}
	\right.
\end{multline}
where $C_*$ is the constant given in \eqref{condi1}.  
Consequently  we use the above estimates and  \eqref{uint+} to compute,   
\begin{multline}\label{spe}
 	\sum_{j=1}^{[m/2]}{{m}\choose j}  \norm{ \partial_x^j u }_{L^\infty} \norm{  \partial_x^{m-j+1}\mathcal U}_{L^2} \\
   \leq   C \sum_{j=4}^{[m/2]}\frac{m!} {j!(m-j)!} \frac{[(j-4)!]^\sigma}{ \rho^{j-4}} \frac{[(m-j-5)!]^\sigma}{\rho^{m-j-5}} \abs{\vec a}_{\rho,\sigma}^2 \\+ C C_*\sum_{1\leq j\leq 3} \frac{m!} {j!(m-j)!}  \frac{[(m-j-5)!]^\sigma}{\tilde\rho^{m-j-5}} \abs{\vec a}_{\tilde \rho,\sigma}. 
\end{multline}
Direct verification shows
 \begin{eqnarray*}
\sum_{1\leq j\leq 3} \frac{m!} {j!(m-j)!}  \frac{[(m-j-5)!]^\sigma}{\tilde\rho^{m-j-5}} \abs{\vec a}_{\tilde \rho,\sigma}  \leq C m  \frac{[(m-6)!]^\sigma}{ \tilde \rho^{m-6}} \abs{\vec a}_{\tilde\rho,\sigma},
 \end{eqnarray*}
and meanwhile  
\begin{eqnarray*}
\begin{aligned}
&  \sum_{j=4}^{[m/2]}\frac{m!} {j!(m-j)!} \frac{[(j-4)!]^\sigma}{ \rho^{j-4}} \frac{[(m-j-5)!]^\sigma}{\rho^{m-j-5}} \abs{\vec a}_{\rho,\sigma}^2 \\
	&\leq  C \frac{   \abs{\vec a}_{\rho,\sigma}^2}{\rho^{m-6}}  \sum_{j=4}^{[m/2]} \frac{m! [(j-4)!]^{\sigma-1} [(m-j-5)!]^{\sigma-1}} {j^4(m-j)^5} \\   
	& \leq   C\frac{   \abs{\vec a}_{\rho,\sigma}^2}{\rho^{m-6}} \sum_{j=4}^{[m/2]} \frac{(m-6)! m^6} {j^4m^5 }   [(m-9)!]^{\sigma-1}\\
	&\leq  C \frac{    [(m-6)!]^{\sigma} \abs{\vec a}_{\rho,\sigma}^2 } {\rho^{m-6}} \frac{ m } {  m^{3(\sigma-1)} } \sum_{j=4}^{[m/2]} \frac{ 1 } {j^4} 
	 \leq  C \frac{  [(m-6)!]^{\sigma} } {\rho^{m-6}}\abs{\vec a}_{\rho,\sigma}^2, 
	 \end{aligned}
\end{eqnarray*}
the  last inequality using the fact that $\sigma\in[3/2,2].$   Combining the above inequalities with \eqref{spe} gives
\begin{eqnarray*}
	\sum_{j=1}^{[m/2]}{{m}\choose j}  \norm{ \partial_x^j u }_{L^\infty} \norm{  \partial_x^{m-j+1}\mathcal U}_{L^2} \leq C \frac{  [(m-6)!]^{\sigma} } {\rho^{m-6}}\abs{\vec a}_{\rho,\sigma}^2+C C_*m  \frac{[(m-6)!]^\sigma}{ \tilde \rho^{m-6}} \abs{\vec a}_{\tilde\rho,\sigma}.
\end{eqnarray*}
  Similarly
\begin{eqnarray*}
\begin{aligned}
	&\sum_{j= [m/2]+1}^m {{m}\choose j}  \norm{ \partial_x^j u}_{L_{x,y}^2(L_z^\infty)}\norm{ \partial_x^{m-j+1} \mathcal U}_{L_{x,y}^\infty(L_z^2)}\\
	& \leq \sum_{j= [m/2]+1}^{m-3} \frac{m!} {j!(m-j)!} \frac{[(j-6)!]^\sigma}{ \rho^{j-6}} \frac{[(m-j-3)!]^\sigma}{\rho^{m-j-3}} \abs{\vec a}_{\rho,\sigma}^2\\
	&\quad+\sum_{j= m-2}^m \frac{m!} {j!(m-j)!} \frac{[(j-6)!]^\sigma}{ \rho^{j-6}}  \abs{\vec a}_{\rho,\sigma}^2\\
	&\leq C \frac{  [(m-6)!]^{\sigma} } {\rho^{m-6}}\abs{\vec a}_{\rho,\sigma}^2.
	\end{aligned}
\end{eqnarray*}
Putting these inequalities into \eqref{feu} we get 
\begin{multline}\label{uem}
	\sum_{j=1}^{m}{{m}\choose j}  \norm{ (\partial_x^j u) \partial_x^{m-j+1}\mathcal U}_{L^2}\\
	\leq C \frac{  [(m-6)!]^{\sigma} } {\rho^{m-6}}\abs{\vec a}_{\rho,\sigma}^2+CC_*  m  \frac{[(m-6)!]^\sigma}{ \tilde \rho^{m-6}} \abs{\vec a}_{\tilde\rho,\sigma}.
\end{multline}
The above estimate also holds with $(\partial_x^j u) \partial_x^{m-j+1}\mathcal U$ replaced by 
	$ (\partial_x^j v) \partial_x^{m-j}\partial_y\mathcal U$.  
 This gives,  using \eqref{uint+} and \eqref{factor},
   \begin{multline*}
   	-\int_0^{t} \Big( \sum_{j=1}^{m}{m\choose j}\Big[(\partial_x^j u) \partial_{x}^{m-j+1}\mathcal U+(\partial_x^j v)\partial_{x}^{m-j}\partial_y\mathcal U  \Big],\ \partial_x^{m} \mathcal U \Big)_{L^2}ds\\
   	\leq  C \frac{    [\inner{m-6}!]^{2\sigma}}{\rho^{2(m-6)}} \bigg( \int_0^{t}   \abs{\vec a(s)}_{ \rho,\sigma}^3  ds+  C_*    \int_0^{t}   \frac{\abs{\vec a(s)}_{\tilde \rho,\sigma}^2}{\tilde\rho-\rho} ds\bigg).
   \end{multline*}
The assertion in Lemma \ref{lemmac} will follow if we have  
     \begin{multline}\label{wes}
   	-\int_0^{t} \Big( \sum_{j=1}^{m}{m\choose j} (\partial_x^j w)\partial_{x}^{m-j}\partial_z\mathcal U,\ \partial_x^{m} \mathcal U \Big)_{L^2}ds \\
   		\leq {1\over 2} \int_0^{t} \norm{\partial_z\partial_x^{m} \mathcal U }_{L^2}^2 ds+  \frac{C   [\inner{m-6}!]^{2\sigma}}{\rho^{2(m-6)}}  \int_0^{t}  \big (\abs{\vec a(s)}_{ \rho,\sigma}^3+   \abs{\vec a(s)}_{ \rho,\sigma}^4  \big)ds.
   \end{multline}
It follows from integration by parts that
\begin{eqnarray}\label{j1j2}
		-\int_0^{t} \Big( \sum_{j=1}^{m}{m\choose j} (\partial_x^j w)\partial_{x}^{m-j}\partial_z\mathcal U,\ \partial_x^{m} \mathcal U \Big)_{L^2}ds\leq J_1+J_2,
\end{eqnarray}
with
\begin{eqnarray*}
J_1&=& \int_0^{t} \sum_{j=1}^{m}{{m}\choose j}\norm{ \big(\partial_x^j w\big)\partial_x^{m-j} \mathcal U}_{L^2}\norm{\partial_z \partial_x^m\mathcal U}_{L^2} ds\\
J_2&=& \int_0^{t} \sum_{j=1}^{m}{{m}\choose j}\norm{ \big(\partial_x^{j+1} u+\partial_x^j\partial_y v\big)\partial_x^{m-j} \mathcal U}_{L^2}\norm{ \partial_x^m\mathcal U}_{L^2} ds.
\end{eqnarray*}
Observe  $
	\norm{\partial^\alpha w}_{L_z^\infty}\leq  C(\norm{\comi z^{\ell}\partial_x\partial^\alpha   u}_{L_z^2}+\norm{ \comi z^{\ell}\partial_y\partial^\alpha  v}_{L_z^2})
$ 
for $\ell>1/2$,    then  it follows \eqref{emix} that  
 \begin{equation} \label{eofw}
 \norm{\comi z^{-\ell}\partial^\alpha w}_{L^2} +\norm{ \partial^\alpha w}_{L_{x,y}^2(L_z^\infty)}   \leq
	\left\{
	\begin{aligned}
	&C \frac{   [\inner{\abs\alpha-6}!]^{ \sigma}}{r^{ (\abs\alpha-6)}}\abs{\vec a}_{r,\sigma},\quad {\rm if}~\abs\alpha \geq 6,\\
	&C\min\Big\{\abs{\vec a}_{r,\sigma},  \  C_*\Big\}, \quad {\rm if}~  \abs\alpha \leq 5.
	\end{aligned}
	\right.
\end{equation} 
Then applying similar argument  for proving \eqref{uem},  we have       
\begin{multline*}
	\sum_{j=1}^{m}{{m}\choose j}\norm{ \big(\partial_x^j w\big)\partial_x^{m-j} \mathcal U}_{L^2}+\sum_{j=1}^{m}{{m}\choose j}\norm{ \big(\partial_x^{j+1} u+\partial_x^j\partial_y v\big)\partial_x^{m-j} \mathcal U}_{L^2}\\ \leq C\frac{    [\inner{m-6}!]^{\sigma}}{\rho^{m-6}} \abs{\vec a}_{\rho,\sigma}^2,
\end{multline*} 
and thus, using the above inequality and \eqref{uint+},
\begin{eqnarray*} 
	J_1 +J_2 \leq  \frac{1}{2}\int_0^{t}\norm{ \partial_z \partial_x^m\mathcal U}_{L^2}^2ds+  C\frac{  [\inner{m-6}!]^{2\sigma}}{\rho^{2(m-6)}}\int_0^{t}  \inner{\abs{\vec a(s)}_{ \rho,\sigma}^3+ \abs{\vec a(s)}_{ \rho,\sigma}^4} ds.
\end{eqnarray*}
This with \eqref{j1j2} yields \eqref{wes}.  
The proof of Lemma \ref{lemmac} is completed.  
\end{proof}

\begin{lemma}\label{lemgm}
Under the same assumption as in Proposition  \ref{prolambda} we have,   for any $m\geq 6,$  any $t\in[0,T]$ and   for  any pair $\inner{\rho,\tilde\rho}$ with  $0<\rho<\tilde\rho< \rho_0\leq 1$, 
	\begin{multline*} 
		\int_0^{t} \Big(\partial_x^{m}\Big[(\partial_x\partial_zu+\partial_y\partial_zv)\int_0^z\mathcal U d\tilde z  +(\partial_xu+\partial_y v)\mathcal U\Big], \partial_x^{m} \mathcal U\Big)_{L^2} ds\\
		\leq      \frac{1}{2}\int_0^{t}\norm{ \partial_z \partial_x^m\mathcal U}_{L^2}^2ds+  C\frac{  [\inner{m-6}!]^{2\sigma}}{\rho^{2(m-6)}}\int_0^{t} \big(  \abs{\vec a(s)}_{ \rho,\sigma}^3+   \abs{\vec a(s)}_{ \rho,\sigma}^4\big)  ds.
	\end{multline*}
	 
\end{lemma} 

\begin{proof}  We only need  to treat the first  term on the left side and use  Leibniz formula and integration by parts   to write
\begin{eqnarray*}
\begin{aligned}
	&\int_0^{t}  \Big(\partial_x^{m}\Big[(\partial_x\partial_zu )\int_0^z\mathcal U d\tilde z   \Big],\ \partial_x^{m} \mathcal U\Big)_{L^2} ds\\
	&=- \int_0^{t}  \Big(\partial_x^{m}\Big[(\partial_xu )\int_0^z\mathcal U d\tilde z   \Big],\ \partial_z \partial_x^{m} \mathcal U\Big)_{L^2} ds
	- \int_0^{t}  \Big(\partial_x^{m}\big[(\partial_xu ) \mathcal U  \big],\   \partial_x^{m} \mathcal U\Big)_{L^2} ds\\
	&\leq I_1+I_2,
	\end{aligned}
\end{eqnarray*}
with
\begin{eqnarray*}
	 I_1&=& 	\int_0^{t} \sum_{j=1}^{[m/2]}{{m}\choose j} \norm{  \comi z^{\ell}\partial_x^j\partial_x  u}_{L_{x,y}^\infty(L_z^2)}\big\|\comi z^{-\ell}\int_0^z\partial_x^{m-j} \mathcal U d\tilde z \big\|_{L_{x,y}^2(L_z^\infty)}\norm{\partial_z\partial_x^m\mathcal U}_{L^2}ds\\ 
	 &&+  	\int_0^{t} \sum_{j=[m/2]+1}^{m } {{m}\choose j} \norm{  \comi z^{\ell}\partial_x^j\partial_x  u}_{L^2}\big\|\comi z^{-\ell}\int_0^z\partial_x^{m-j} \mathcal U d\tilde z \big\|_{L^\infty} \norm{\partial_z\partial_x^m\mathcal U}_{L^2} ds,\\
	 I_2&=& \int_0^{t}  \sum_{0\leq j \leq [m/2]}{{m}\choose j} \norm{ \partial_x^j\partial_x  u}_{L_{x,y}^\infty(L_z^2)}\big\| \partial_x^{m-j} \mathcal U   \big\|_{L^2}\norm{ \partial_x^m\mathcal U}_{L_{x,y}^2(L_z^\infty)} ds   \\
	  &&+ \int_0^{t}   \sum_{[m/2]+1\leq j \leq m } {{m}\choose j} \norm{   \partial_x^j\partial_x u}_{L^2}\big\| \partial_x^{m-j} \mathcal U  \big\|_{L_{x,y}^\infty(L_z^2)}\norm{ \partial_x^m\mathcal U}_{L_{x,y}^2(L_z^\infty)} ds.
	\end{eqnarray*}
 Now we follow the similar argument as in the proof of Lemma \ref{lemmac}, 
  using  the estimates \eqref{uint+} and  \eqref{emix} as well as  the Sobolev inequality \eqref{soblev},  to compute
\begin{eqnarray*}\begin{aligned}
	&\sum_{0\leq j \leq [m/2]}{{m}\choose j} \norm{  \comi z^{\ell}\partial_x^j\partial_x  u}_{L_{x,y}^\infty(L_z^2)}\big\|\comi z^{-\ell}\int_0^z\partial_x^{m-j} \mathcal U d\tilde z \big\|_{L_{x,y}^2(L_z^\infty)}\\
	&\leq     C   \sum_{j=4}^{  [m/2]} \frac{m!} {j!(m-j)!} \frac{[(j-4)!]^\sigma}{\rho^{j-4}} \abs{\vec a}_{\rho,\sigma}\times \frac{[(m-j-6)!]^\sigma}{\rho^{m-j-6}} \abs{\vec a}_{\rho,\sigma} \\
	&\quad + C   \sum_{0\leq j\leq 3} \frac{m!} {j!(m-j)!}  \abs{\vec a}_{\rho,\sigma}\times \frac{[(m-j-6)!]^\sigma}{\rho^{m-j-6}} \abs{\vec a}_{\rho,\sigma}  \\
	&\leq C  \frac{\abs{\vec a}_{\rho,\sigma}^2 }{\rho^{m-6}}  \sum_{j=4}^{  [m/2]} \frac{ (m-6)! m^6 [(j-4)!]^{\sigma-1} [(m-j-6)!]^{\sigma-1} } {j^4(m-j)^6} +C  \frac{   [(m-6)!]^{\sigma}  }{ \rho^{m-6}}  \abs{\vec a}_{\rho,\sigma}^2\\
	&\leq  C  \frac{ [(m-6)!]^{\sigma}  }{\rho^{m-6}} \abs{\vec a}_{\rho,\sigma}^2.
	\end{aligned}
\end{eqnarray*}
 Similarly,
 \begin{eqnarray*}
 	\sum_{[m/2]+1\leq j \leq m } {{m}\choose j} \norm{  \comi z^{\ell}\partial_x^j\partial_x  u}_{L^2}\big\|\comi z^{-\ell}\int_0^z\partial_x^{m-j} \mathcal U d\tilde z \big\|_{L^\infty} \leq  C  \frac{ [(m-6)!]^{\sigma}  }{\rho^{m-6}} \abs{\vec a}_{\rho,\sigma}^2.
 \end{eqnarray*}
 Thus
 \begin{eqnarray*}
 	I_1\leq   \frac{1}{8}\int_0^{t}\norm{ \partial_z \partial_x^m\mathcal U}_{L^2}^2ds+  C\frac{  [\inner{m-6}!]^{2\sigma}}{\rho^{2(m-6)}}\int_0^{t}    \abs{\vec a(s)}_{ \rho,\sigma}^4  ds.
 \end{eqnarray*}
Observe  $\norm{ \partial_x^m\mathcal U}_{L_{x,y}^2(L_z^\infty)}\leq  C\norm{ \partial_x^m\mathcal U}_{^2}+C\norm{ \partial_z\partial_x^m\mathcal U}_{L^2}$. Then following  the  argument for treating $I_1$ we have also
\begin{eqnarray*}
	I_2\leq   \frac{1}{8}\int_0^{t}\norm{ \partial_z \partial_x^m\mathcal U}_{L^2}^2ds+  C\frac{  [\inner{m-6}!]^{2\sigma}}{\rho^{2(m-6)}}\int_0^{t} \big(  \abs{\vec a(s)}_{ \rho,\sigma}^3+   \abs{\vec a(s)}_{ \rho,\sigma}^4\big)  ds.
\end{eqnarray*}
 Then
 \begin{multline*}
	\int_0^{t}  \Big(\partial_x^{m}\Big[(\partial_x\partial_zu )\int_0^z\mathcal U d\tilde z   \Big],\ \partial_x^{m} \mathcal U\Big)_{L^2} ds\leq I_1+I_2\\
	\leq  \frac{1}{4}\int_0^{t}\norm{ \partial_z \partial_x^m\mathcal U}_{L^2}^2ds+  C\frac{  [\inner{m-6}!]^{2\sigma}}{\rho^{2(m-6)}}\int_0^{t} \big(  \abs{\vec a(s)}_{ \rho,\sigma}^3+   \abs{\vec a(s)}_{ \rho,\sigma}^4\big)  ds.
\end{multline*}
Just following the argument above with slight modification,  the other terms can be controlled by the same upper bound as above, 
 The proof of Lemma \ref{lemgm} is completed.  
\end{proof}	

 \begin{proof}
 	[Completion of the proof of Proposition \ref{prolambda}]
 	Now we  put  the estimates in Lemmas \ref{lemhig}-\ref{lemgm}  into \eqref{uma+} to obtain,   for any  $ m\geq 6,$   any $t\in[0,T]$ and    any pair $\inner{\rho,\tilde\rho}$ with  $0<\rho<\tilde\rho< \rho_0\leq 1$,   
\begin{eqnarray*}
	\norm{\partial_x^{m} \mathcal U(t)}_{L^2}^2 
	\leq  C\frac{ [\inner{m-6}!]^{2\sigma}}{\rho^{2(m-6)}} \bigg( \int_0^{t}  \big(\abs{\vec a(s)}_{ \rho,\sigma}^2+   \abs{\vec a(s)}_{ \rho,\sigma}^4\big)ds+C_*  \int_0^{t}   \frac{\abs{\vec a(s)}_{\tilde \rho,\sigma}^2}{\tilde\rho-\rho}ds\bigg).
\end{eqnarray*}
Similarly the above estimate also holds with $\partial_x^m$ replaced by $\partial_y^m.$  Thus by \eqref{realp} we have,  any $t\in[0,T]$ and   for  any pair $\inner{\rho,\tilde\rho}$ with  $0<\rho<\tilde\rho< \rho_0\leq 1$, 
\begin{multline*}
		 	\sup_{\abs\alpha\geq 6}\frac{\rho^{2(\abs\alpha-6)}}{   [\inner{\abs\alpha-6}!]^{2\sigma}}   \norm{\partial^{\alpha} \mathcal U(t)}_{L^2}^2  \\
		 	 \leq C        \int_0^{t}   \big( \abs{\vec a(s)}_{ \rho,\sigma}^2+\abs{\vec a(s)}_{ \rho,\sigma}^4\big)   \,ds+  CC_*  \int_0^{t}  \frac{  \abs{\vec a(s)}_{ \tilde\rho,\sigma}^2}{\tilde\rho-\rho}\,ds.
\end{multline*} 
It can be checked
 straightforwardly
that the same  upper bound holds for $\norm{\partial^{\alpha} \mathcal U}_{L^2}$  with $\abs\alpha\leq 5.$    Then the desired estimate for $\partial^{\alpha} \mathcal U$ in Proposition \ref{prolambda} follows, and similarly for $\partial^{\alpha} \widetilde{ \mathcal U}$.   The proof  of Proposition \ref{prolambda} is thus completed.
 \end{proof}	 	
 	 	
 	 	\section{Estimate on $\norm{(u,v)}_{\rho,\sigma}$}\label{secofuv}
 	
 The main estimate on $\norm{(u,v)}_{\rho,\sigma}$ can be stated as follows, recalling  $\norm{ (u,v\big)}_{\rho,\sigma}$  is given in Definition \ref{defgev}.  
 	 	
\begin{proposition}\label{propuv++++}
 Under Assumption \ref{assmain} we have,  
 	  for any    $t\in[0,T]$ and  
  any pair $\inner{\rho,\tilde\rho}$ with $0<\rho<\tilde\rho<\rho_0\leq 1$,    
\begin{multline*}  
  \norm{\big(u(t),v(t)\big)}_{\rho,\sigma}^2   \leq   C_1\norm{(u_0,v_0)}_{2\rho_0,\sigma}^2\\
  + C  C_*^3\bigg(   \int_0^{t}  \big(\abs{\vec a(s)}_{\rho,\sigma}^2+\abs{\vec a(s)}_{\rho,\sigma}^4 \big)ds
 		+   	\int_0^{t}   \frac{ \abs{\vec a(s)}_{\tilde\rho,\sigma}^2}{\tilde\rho-\rho} ds\bigg),
\end{multline*}
where $C_*\geq 1 $ is the constant given in \eqref{condi1}.
 \end{proposition}
  
  In view of Definition \ref{defgev} of $\norm{(u,v)}_{\rho,\sigma},$ the above proposition will follow from the  two lemmas as below.  
  
 \begin{lemma}[Estimate on the tangential derivatives] \label{lemtan}
  Under the same assumption as in Proposition \ref{propuv++++} we have, recalling $\partial^\alpha=\partial_x^{\alpha_1}\partial_y^{\alpha_2}$, 
 	 \begin{multline*}	
\sup_{\abs\alpha \geq 7}  \frac{\rho^{2(\abs\alpha-7)}} {  [(\abs\alpha-7)!]^{2\sigma}  } \norm{\comi z^{\ell}\partial^\alpha u (t)}_{L^2}^2 
 + \sup_{\abs\alpha \leq 6}    \norm{\comi z^{\ell}\partial^\alpha u (t)}_{L^2}^2 \\	
 		 \leq  C_1 \norm{(u_0,v_0)}_{2\rho_0,\sigma}^2 + C  C_*^3\bigg(   \int_0^{t}  \big(\abs{\vec a(s)}_{\rho,\sigma}^2+\abs{\vec a(s)}_{\rho,\sigma}^4 \big)ds
 		+   	\int_0^{t}   \frac{ \abs{\vec a(s)}_{\tilde\rho,\sigma}^2}{\tilde\rho-\rho} ds\bigg).
 	\end{multline*}
 	Similarly the upper bound still holds with  $\partial^\alpha u$ replaced by $\partial^\alpha v$. 
 \end{lemma}

\begin{proof}
 We need only to estimate $u$ since $v$ can be treated in the same way.   
Applying $\partial_x^m$ to the first equation in \eqref{prandtl}   gives
  \begin{eqnarray}\label{zu}
\big(\partial_t+u\partial_x+v\partial_y+w\partial_z-\partial_z^2\big)  \partial_x^m  u = -(\partial_x^mw)\partial_z u +F_{m}
    	\end{eqnarray}
    	with
    	\begin{equation*} 
		F_{m}=-\sum_{j=1}^{m}{{m}\choose j} \Big[(\partial_x^j u) \partial_x^{m-j+1} u+ \big(\partial_x^j v\big)\partial_x^{m-j}\partial_y u \Big]-\sum_{j=1}^{m-1}{{m}\choose j}  (\partial_x^j w) \partial_x^{m-j} \partial_z u.
	\end{equation*}
	On the other hand,  applying $(\partial_z u)\partial_x^{m-1}$ to \eqref{mau}   we have
	\begin{eqnarray}\label{eqint}
		\big(\partial_t+u\partial_x+v\partial_y+w\partial_z-\partial_z^2\big) (\partial_z u) \int_0^z \partial_x^{m-1} \mathcal U  d\tilde z=- (\partial_x^mw) \partial_z u+L_m
	\end{eqnarray}
	with 
	\begin{equation*}
	\begin{aligned}
		 L_m=&-(\partial_z u) \sum_{j=1}^{m-1}{{m-1}\choose j} \Big[(\partial_x^j u) \int_0^z \partial_x^{m-j} \mathcal U d\tilde z+ \big(\partial_x^j v\big) \int_0^z \partial_x^{m-1-j}\partial_y \mathcal U d\tilde z    \Big] \\
		&-(\partial_z u) \sum_{j=1}^{m-1}{{m-1}\choose j}  (\partial_x^j w) \partial_{x}^{m-1-j} \mathcal U   \\
		& +\Big[(\partial_y v)\partial_zu- ( \partial_yu)\partial_zv\Big]\int_0^z \partial_x^{m-1} \mathcal U  d\tilde z-2(\partial_z^2 u)\partial_x^{m-1} \mathcal U,
		\end{aligned}
		\end{equation*}
	where we have used the fact that, denoting by  $[T_1,T_2]=T_1T_2-T_2T_1$  the commutator of two operators $T_1, T_2,$ 
	\begin{multline*}
	\big[	\partial_t+u\partial_x+v\partial_y+w\partial_z-\partial_z^2, \  (\partial_z u)\big]\\=\big(\partial_t+u\partial_x+v\partial_y+w\partial_z-\partial_z^2\big)\partial_zu-2(\partial_z^2u)\partial_z 
	=(\partial_y v)\partial_zu- ( \partial_yu)\partial_zv-2(\partial_z^2u)\partial_z.
	\end{multline*}
 Now we subtract the  equation \eqref{eqint} by \eqref{zu} to eliminate the higher order term $(\partial_x^mw)\partial_zu$ and  this   gives the equation for
 \begin{eqnarray}\label{plam}
 	\psi_m=  \partial_x^m u-(\partial_zu)\int_0^z\partial_x^{m-1} \mathcal U d\tilde z ; 
 \end{eqnarray}
 that is, 
 \begin{equation} \label{eqp}
 		\big(\partial_t+u\partial_x+v\partial_y+w\partial_z-\partial_z^2\big) \psi_m 
 		=  F_{m}-L_m,
 \end{equation}
 and thus
 \begin{multline*}
 	\big(\partial_t+u\partial_x+v\partial_y+w\partial_z-\partial_z^2\big) \comi z^\ell\psi_m \\
 		= \comi z^\ell F_{m}-\comi z^\ell L_m+w(\partial_z \comi z^\ell) \psi_m-(\partial_z^2 \comi z^\ell) \psi_m-2(\partial_z\comi z^\ell)\partial_z\psi_m.
 \end{multline*}
Then we take the scalar product with $ \comi z^\ell \psi_m$ and 
observe $ \comi z^\ell\psi_m|_{z=0}=0,$ to obtain 
 \begin{equation}\label{dps}
 \begin{aligned}
 &\frac{1}{2} \norm{\comi z^{\ell}\psi_m(t)}_{L^2}^2-\frac{1}{2} \norm{\comi z^{\ell}\psi_m(0)}_{L^2}^2+\int_0^{t} \big\|\partial_z\big[\comi z^{\ell}\psi_m(s)\big]\big\|_{L^2}^2 ds\\
& = \int_0^{t} \Big(\big(\partial_s+u\partial_x+v\partial_y+w\partial_z-\partial_z^2\big) \comi z^\ell\psi_m,\  \comi z^{\ell}\psi_m\Big)_{L^2}ds\\
& =\int_0^{t} \Big(\comi z^\ell F_{m},  \ \comi z^{\ell}\psi_m\Big)_{L^2}ds-\int_0^{t} \Big(\comi z^\ell L_m,\  \comi z^{\ell}\psi_m\Big)_{L^2}ds\\
&\quad+\int_0^{t} \Big( w(\partial_z \comi z^\ell) \psi_m-(\partial_z^2 \comi z^\ell) \psi_m-2(\partial_z\comi z^\ell)\partial_z\psi_m, \  \comi z^{\ell}\psi_m\Big)_{L^2}ds.
 \end{aligned}
 \end{equation}
 Note for any $0<r\leq \rho_0$ we have, observing $C_*\geq 1,$
 \begin{equation}\label{pm+}
	\norm{\comi z^{\ell}\psi_m}_{L^2}\leq  \norm{\comi z^{\ell}\partial_x^m u}_{L^2}+CC_*\norm{\partial_x^{m-1}\mathcal U}_{L^2} \leq CC_* \frac{   [(m-7)!]^{\sigma}  } {r^{m-7}} \abs{\vec a}_{r,\sigma}
\end{equation} 
due to  the definition of $\psi_m$ given by  \eqref{plam} as well as \eqref{emix}, \eqref{uint+} and \eqref{condi1}. Then, in view of \eqref{eofw},   
\begin{multline}\label{rigs2}
	 \int_0^{t} \Big( w(\partial_z \comi z^\ell) \psi_m-(\partial_z^2 \comi z^\ell) \psi_m-2(\partial_z\comi z^\ell)\partial_z\psi_m, \  \comi z^{\ell}\psi_m\Big)_{L^2}ds\\
   \leq  C C_*^3\frac{    [(m-7)!]^{2\sigma} } {\rho^{2(m-7)}}\int_0^{t}  \abs{\vec a(s)}_{\rho,\sigma}^2ds. 
\end{multline}
Note $F_m$ is given in \eqref{zu}, and we apply similar computation as in the proof of Lemma \ref{lemmac}, using \eqref{emix} instead of \eqref{uint+}; this yields
 \begin{multline}\label{righs}
 	\int_0^{t} \Big(\comi z^\ell F_{m}, \  \comi z^{\ell}\psi_m\Big)_{L^2}ds  \\
 	 \leq  CC_* ^2\frac{    [(m-7)!]^{2\sigma} } {\rho^{2(m-7)}}\bigg(\int_0^{t}   \abs{\vec a(s)}_{\rho,\sigma}^3  ds + \int_0^{t} \frac{\abs{\vec a(s)}_{\tilde\rho,\sigma}^2}{\tilde\rho-\rho}ds\bigg).
 \end{multline}
 Recall $L_m$ is given in \eqref{eqint}.  Then  
\begin{eqnarray*}
\norm{\comi z^{\ell} L_m} &\leq& 	\norm{\comi z^{\ell+1}\partial_z u }_{L^\infty}\sum_{j=1}^{m-1}{{m-1}\choose j}  \big\|\comi z^{-1}(\partial_x^j u) \int_0^z \partial_x^{m-j} \mathcal U d\tilde z\big\|_{L^2} \\
&&+	\norm{\comi z^{\ell+1}\partial_z u }_{L^\infty}\sum_{j=1}^{m-1}{{m-1}\choose j}  \big\|\comi z^{-1}(\partial_x^j v) \int_0^z \partial_x^{m-1-j}\partial_y \mathcal U d\tilde z\big\|_{L^2}\\
&&+	\norm{\comi z^{\ell+1}\partial_z u }_{L^\infty}\sum_{j=1}^{m-1}{{m-1}\choose j} \Big[\big\|\comi z^{-1}(\partial_x^j w)   \partial_x^{m-1-j} \mathcal U  \big\|_{L^2}\\
&&
		+\Big\|\comi z^{2\ell+1}\Big[(\partial_y v)\partial_zu- ( \partial_yu)\partial_zv\Big]\Big\|_{L^\infty}\big\|\comi z^{-\ell-1} \int_0^z \partial_x^{m-1} \mathcal U  d\tilde z\big\|_{L^2}\\
		&&+2\norm{\comi z^{\ell}\partial_z^2 u}_{L^\infty}\norm{\partial_x^{m-1} \mathcal U}_{L^2}.
\end{eqnarray*}
 Thus we apply again the argument for proving Lemma \ref{lemmac} to obtain, observing \eqref{condi1} and using \eqref{uint+},  \eqref{emix}    and \eqref{eofw} ,
\begin{eqnarray*} 
	\norm{\comi z^{\ell} L_m}\leq   C \frac{    [(m-7)!]^{\sigma} } {\rho^{m-7}}\inner{\abs{\vec a}_{\rho,\sigma}^2+\abs{\vec a}_{\rho,\sigma}^3}  + C  C_*^2  m\frac{[(m-7)!]^\sigma}{\tilde\rho^{m-7}} \abs{\vec a}_{\tilde\rho,\sigma},
\end{eqnarray*} 
and thus, with \eqref{pm+} and \eqref{factor},
\begin{multline*}  
 \int_0^{t} \Big(\comi z^\ell L_{m}, \  \comi z^{\ell}\psi_m\Big)_{L^2}ds  \\
 \leq  CC_*^3 \frac{    [(m-7)!]^{2\sigma} } {\rho^{2(m-7)}}\bigg(\int_0^{t} \inner{\abs{\vec a}_{\rho,\sigma}^3+\abs{\vec a}_{\rho,\sigma}^4}ds  + \int_0^{t} \frac{\abs{\vec a(s)}_{\tilde\rho,\sigma}^2}{\tilde\rho-\rho}ds\bigg).
 \end{multline*}
Putting the above estimate and the estimates \eqref{rigs2} and \eqref{righs} into \eqref{dps}, yields
\begin{multline*}
 		\norm{\comi z^{\ell} \psi_m(t)}_{L^2}^2+\int_0^t \norm{\partial_z\big(\comi z^{\ell} \psi_m(s)\big)}_{L^2}^2 dt  
 		\leq   \norm{\comi z^\ell  \psi_m (0)}_{L^2}^2 \\+CC_*^3   \frac{   [(m-7)!]^{2\sigma}  }{\rho^{2(m-7)}}	\bigg(\int_0^{t}  \big(\abs{\vec a(s)}_{\rho,\sigma}^2+\abs{\vec a(s)}_{\rho,\sigma}^4 \big)ds 
 		+    \int_0^{t}   \frac{ \abs{\vec a(s)}_{\tilde\rho,\sigma}^2}{\tilde\rho-\rho} ds\bigg) .
 	\end{multline*}
Moreover 	observe  
 $
 	\comi z^{\ell} \psi_m|_{t=0}=\comi z^{\ell} \partial_x^{m}u_0 
$
and thus
\begin{eqnarray*}
 \norm{\comi z^\ell  \psi_m (0)}_{L^2}^2  \leq   \frac{[(m-7)!]^{2\sigma}  }{  (2\rho_0)^{2(m-7)}} \norm{(u_0,v_0)}_{2\rho_0,\sigma}^2\leq \frac{ [(m-7)!]^{2\sigma}  }{  \rho^{2(m-7)}} \norm{(u_0,v_0)}_{2\rho_0,\sigma}^2.
\end{eqnarray*} 
Then we obtain 
 \begin{multline} \label{epsm}
	\norm{\comi z^{\ell} \psi_m(t)}_{L^2}^2 +\int_0^t \norm{\partial_z\big(\comi z^{\ell} \psi_m(s)\big)}_{L^2}^2 dt
 		\leq   		 \frac{  [(m-7)!]^{2\sigma}  } {\rho^{2(m-7)}}\norm{(u_0,v_0)}_{2\rho_0,\sigma}^2 \\+ C  C_*^3 		 \frac{  [(m-7)!]^{2\sigma}  } {\rho^{2(m-7)}} \bigg(    \int_0^{t}  \big(\abs{\vec a(s)}_{\rho,\sigma}^2+\abs{\vec a(s)}_{\rho,\sigma}^4 \big)ds
 		+  	\int_0^{t}   \frac{ \abs{\vec a(s)}_{\tilde\rho,\sigma}^2}{\tilde\rho-\rho} ds\bigg).
 	\end{multline}
 Consequently this inequality, along with
 the estimate 
 	\begin{eqnarray*}
 		\norm{\comi z^{\ell}\partial_x^m u}_{L^2}^2\leq 2 \big\|\comi z^{\ell}\psi_m\big\|_{L^2}^2+2\big\|\comi z^{\ell}(\partial_zu)\int_0^z\partial_x^{m-1} \mathcal U d\tilde z\big\|_{L^2}^2
 	\end{eqnarray*}
 that are from the definition \eqref{plam} of $\psi_m,$ and the fact that 
 	\begin{eqnarray*}
 	\begin{aligned}
 		&\big\|\comi z^{\ell}(\partial_zu(t))\int_0^z\partial_x^{m-1} \mathcal U(t) d\tilde z\big\|_{L^2}^2 \\
 		& \leq 
 	\big\|\comi z^{\ell+1} \partial_zu (t)\big\|_{L_{x,y}^\infty(L_z^2)}^2\big\|\comi z^{-1} \int_0^z\partial_x^{m-1} \mathcal U(t) d\tilde z\big\|_{L_{x,y}^2(L_z^\infty)}^2\\
 	&\leq C C_*^2\norm{\partial_x^{m-1}\mathcal U(t)}_{L^2}^2 \\
 	&\leq C C_*^3\frac{  [(m-7)!]^{2\sigma}  } {\rho^{2(m-7)}} \bigg(  \int_0^{t}  \big(\abs{\vec a(s)}_{\rho,\sigma}^2+\abs{\vec a(s)}_{\rho,\sigma}^4 \big)ds
 		+   	\int_0^{t}   \frac{ \abs{\vec a(s)}_{\tilde\rho,\sigma}^2}{\tilde\rho-\rho} ds\bigg)
 		\end{aligned}
 	\end{eqnarray*}
 	due to \eqref{condi1} and Proposition \ref{prolambda},  yields that for any $m\geq 7$ and any $t\in[0,T],$
 	 \begin{multline*}	
   	\norm{\comi z^{\ell}\partial_x^m u (t)}_{L^2}^2  		 \leq 2\frac{  [(m-7)!]^{2\sigma}  } {\rho^{2(m-7)}} \norm{(u_0,v_0)}_{2\rho_0,\sigma}^2\\
   	+C  C_*^3  \frac{  [(m-7)!]^{2\sigma}  } {\rho^{2(m-7)}} \bigg(  \int_0^{t}  \big(\abs{\vec a(s)}_{\rho,\sigma}^2+\abs{\vec a(s)}_{\rho,\sigma}^4 \big)ds
 		+    	\int_0^{t}   \frac{ \abs{\vec a(s)}_{\tilde\rho,\sigma}^2}{\tilde\rho-\rho} ds\bigg).
 	\end{multline*}
 	 We have proven the assertion for $\partial^\alpha=\partial_x^m$ with $m\geq 7.$
 	By direct verification we can get the desired estimate for  $m\leq 6$.  Then we have obtained the estimate as desired for $\partial^\alpha=\partial_x^m$.    Moreover 
 	the above estimates also hold with $\partial_x^m$ replaced by $\partial_y^m,$ following the  similar argument.  Thus the desired estimate for general $\partial^\alpha u$ follows in view of \eqref{realp}. Similarly for  $\partial^\alpha v$. 
 	 The proof of Lemma \ref{lemtan} is completed.  
\end{proof}

  \begin{lemma}[Estimate on the mixed derivatives] \label{lemnor}
  Under the same assumption as in Proposition \ref{propuv++++} we have
 	 \begin{multline*}	
  \sup_{\stackrel{1\leq j\leq 5}{ \abs\alpha+j \geq 7}}  \frac{\rho^{2(\abs\alpha+j-7)}} {  [(\abs\alpha+j-7)!]^{2\sigma}  } 	\norm{\comi z^{\ell+j}\partial^\alpha \partial_z^ju (t)}_{L^2}^2
  +\sup_{\stackrel{1\leq j\leq 5} {\abs\alpha +j \leq 6}}  \norm{\comi z^{\ell+j}\partial^\alpha \partial_z^ju (t)}_{L^2}^2 \\	
 		 \leq  C_1\norm{(u_0,v_0)}_{2\rho_0,\sigma}^2+C     \int_0^{t}  \big(\abs{\vec a(s)}_{\rho,\sigma}^2+\abs{\vec a(s)}_{\rho,\sigma}^4 \big)ds.
 	\end{multline*}
 	Similarly for $\partial^\alpha v.$ 
 \end{lemma}
 
 \begin{proof}  The upper bound for   $ \norm{\comi z^{\ell+j}\partial^\alpha \partial_z^ju (t)}_{L^2}$ with   $\abs\alpha +j \leq 6  $   and $ 1\leq j\leq 5$ is straightforward.  So we only need to consider the case of $\abs\alpha +j \geq 7  $  with $ 1\leq j\leq 5$. As before it suffices to estimate $ \comi z^{\ell+j}\partial_x^m \partial_z^ju$ since $ \comi z^{\ell+j}\partial_y^m \partial_z^ju$ can be treated in the same way.  
    
    We apply $\comi z^{\ell+j}\partial_z^j$ to equation \eqref{zu} to get 
 \begin{multline*}
\big(\partial_t+u\partial_x+v\partial_y+w\partial_z-\partial_z^2\big)  \comi z^{\ell+j}  \partial_x^m  \partial_z^j u\\
 = -  \comi z^{\ell+j} \partial_z^j\big[ (\partial_x^mw)\partial_z u \big]+ \comi z^{\ell+j} \partial_z^j F_{m}+\big[u\partial_x+v\partial_y+w\partial_z-\partial_z^2,  \ \comi z^{\ell+j}\partial_z^j \big]  \partial_x^m u,
 \end{multline*}
 where
  $F_m$ is defined in \eqref{zu} and $[T_1,T_2]=T_1T_2-T_2T_1$ stands for the commutator of two operators $T_1, T_2$. Thus,
 \begin{equation}\label{em}
 \begin{aligned}
 	&\frac{1}{2}\frac{d}{dt}\norm{\comi z^{\ell+j}  \partial_x^m  \partial_z^j u}_{L^2}^2+\big\|\partial_z\big(\comi z^{\ell+j}  \partial_x^m  \partial_z^j u\big)\big\|_{L^2}^2\\
 	&\qquad+\int_{\mathbb R^2} \big(  \partial_x^m  \partial_z^j u\big)   \big(\partial_x^m  \partial_z^{j+1} u\big)\big|_{z=0}  dxdy\\
& =\Big(\big(\partial_t+u\partial_x+v\partial_y+w\partial_z-\partial_z^2\big)  \comi z^{\ell+j}  \partial_x^m  \partial_z^j u,\ \comi z^{\ell+j}  \partial_x^m  \partial_z^j u\Big)_{L^2}\\
& =\Big(- \comi z^{\ell+j} \partial_z^j\big[ (\partial_x^mw)\partial_z u \big]+\comi z^{\ell+j} \partial_z^j F_{m},\ \comi z^{\ell+j}  \partial_x^m  \partial_z^j u\Big)_{L^2}\\
& \quad+\Big(   \big[u\partial_x+v\partial_y+w\partial_z-\partial_z^2,  \ \comi z^{\ell+j}\partial_z^j \big]  \partial_x^m u,\ \comi z^{\ell+j}  \partial_x^m  \partial_z^j u\Big)_{L^2},
\end{aligned}
 \end{equation}
 where we used the fact that $$\big(\comi z^{\ell+j}  \partial_x^m  \partial_z^j u\big)\partial_z\big(\comi z^{\ell+j}  \partial_x^m  \partial_z^j u\big)\big|_{z=0}=\big(  \partial_x^m  \partial_z^j u\big)   \big(\partial_x^m  \partial_z^{j+1} u\big)\big|_{z=0}.$$
As for the terms on the right side of   \eqref{em} we use the argument for proving Lemma \ref{lemmac} to get, recalling $F_m$ is given in \eqref{zu}, 
  \begin{multline*}
  	\Big( \comi z^{\ell+j} \partial_z^j F_{m} ,\ \comi z^{\ell+j}  \partial_x^m  \partial_z^j u\Big)_{L^2}\\
  	\leq \frac{1}{8}  \big\|\partial_z\big(\comi z^{\ell+j}  \partial_x^m  \partial_z^j u\big)\big\|_{L^2}^2+C\big\|\comi z^{\ell+j} \partial_z^{j-1} F_{m} \big\|_{L^2}^2+C\big\|\comi z^{\ell+j}  \partial_x^m  \partial_z^j u \big\|_{L^2}^2 \\
  	\leq  \frac{1}{8}  \big\|\partial_z\big(\comi z^{\ell+j}  \partial_x^m  \partial_z^j u\big)\big\|_{L^2}^2+ C \frac{    [(m+j-7)!]^{2\sigma} } {\rho^{2(m+j-7)}} \inner{\abs{\vec a}_{\rho,\sigma}^2+\abs{\vec a}_{\rho,\sigma}^4}.
  \end{multline*}
 Moreover, direct verification shows
   \begin{eqnarray*}
   \begin{aligned}
  &	\Big(- \comi z^{\ell+j}  \partial_z^j\big[ (\partial_x^mw)\partial_z u \big] +\big[u\partial_x+v\partial_y+w\partial_z-\partial_z^2,  \ \comi z^{\ell+j}\partial_z^j \big]  \partial_x^m u, \comi z^{\ell+j}  \partial_x^m  \partial_z^j u\Big)_{L^2}\\
& \leq 	 \frac{1}{8}  \big\|\partial_z\big(\comi z^{\ell+j}  \partial_x^m  \partial_z^j u\big)\big\|_{L^2}^2+ C \frac{    [(m+j-7)!]^{2\sigma} } {\rho^{2(m+j-7)}} \inner{\abs{\vec a}_{\rho,\sigma}^2+\abs{\vec a}_{\rho,\sigma}^4}. 
\end{aligned}
  \end{eqnarray*}
By   the two inequalities above we get   the upper bound for the terms on the right side of  \eqref{em}, that is, 
 \begin{multline*} 
\Big(- \comi z^{\ell+j} \partial_z^j\big[ (\partial_x^mw)\partial_z u \big]+\comi z^{\ell+j} \partial_z^j F_{m},\ \comi z^{\ell+j}  \partial_x^m  \partial_z^j u\Big)_{L^2}\\
 +\Big(  \big[u\partial_x+v\partial_y+w\partial_z-\partial_z^2,  \ \comi z^{\ell+j}\partial_z^j \big]  \partial_x^m u,\ \comi z^{\ell+j}  \partial_x^m  \partial_z^j u\Big)_{L^2}\\
 	\leq \frac{1}{4}\big\|\partial_z\big(\comi z^{\ell+j}  \partial_x^m  \partial_z^j u\big)\big\|_{L^2}^2+ C\frac{   [\inner{m+j-7}!]^{2\sigma}}{\rho^{2(m+j-7)}}   \big(\abs{\vec a}_{ \rho,\sigma}^2+   \abs{\vec a}_{ \rho,\sigma}^4 \big).
 \end{multline*}  
 This  with \eqref{em}  yields
 \begin{multline}\label{efei}
 \frac{1}{2}\frac{d}{dt}\norm{\comi z^{\ell+j}  \partial_x^m  \partial_z^j u}_{L^2}^2+\frac{3}{4}\big\|\partial_z\big(\comi z^{\ell+j}  \partial_x^m  \partial_z^j u\big)\big\|_{L^2}^2 \\
 \leq \Big|\int_{\mathbb R^2} \big(  \partial_x^m  \partial_z^j u\big)   \big(\partial_x^m  \partial_z^{j+1} u\big)\big|_{z=0}  dxdy\Big| + C\frac{   [\inner{m+j-7}!]^{2\sigma}}{\rho^{2(m+j-7)}}   \big(\abs{\vec a}_{ \rho,\sigma}^2+   \abs{\vec a}_{ \rho,\sigma}^4 \big).
 \end{multline}
Next we handle the first term on the right of \eqref{efei}.  We claim the following estimate holds for $j=1, 2,3,5,$
    \begin{multline}\label{eb}
 	\bigg|\int_{\mathbb R^2}  \big( \partial_x^m  \partial_z^j u\big)\big(   \partial_x^m  \partial_z^{j+1} u\big)\big|_{z=0}dxdy\bigg|\\ 
 	\leq \frac{1}{4}\big\|  \partial_x^m  \partial_z^{j+1} u \big\|_{L^2}^2+C \frac{    [\inner{m+j-7}!]^{2\sigma}}{\rho^{2(m+j-7)}}   \big(\abs{\vec a }_{ \rho,\sigma}^2+   \abs{\vec a }_{ \rho,\sigma}^4 \big).
 \end{multline} 
      Observe
 $(\partial_{z}^2u, \partial_{z}^2v)|_{z=0}=(0,0)$ which follows after taking trace for the equations in \eqref{prandtl}.  Moreover applying  
 $\partial_z^2$ to the first equation in \eqref{prandtl} gives
 \begin{eqnarray*}
\partial_t \partial_z^2 u+\partial_z^2\inner{u\partial_x u+v\partial_y u+w\partial_z u}-\partial_z^4 u=0
 \end{eqnarray*}
 and thus,
 \begin{multline}\label{bo4}
 	\partial_z^4 u|_{z=0}=\partial_z^2\inner{u\partial_x u+v\partial_y u+w\partial_z u}|_{z=0}\\
 	= (\partial_z u)\inner{\partial_x \partial_z u-\partial_y\partial_z v} |_{z=0}+2 (\partial_z v) \partial_y\partial_z u|_{z=0}.
 \end{multline}
 Moreover the above relationship, along with the equation
 \begin{eqnarray*}
 	\partial_t \partial_z^4 u+\partial_z^4\inner{u\partial_x u+v\partial_y u+w\partial_z u}-\partial_z^6 u=0,
 \end{eqnarray*}
 yields
  \begin{eqnarray*}
 	\partial_z^6 u|_{z=0}=- (\partial_z^3 u)\inner{\partial_x \partial_z u+\partial_y\partial_z v} |_{z=0}+4 (\partial_z u) \partial_x \partial_z^3 u|_{z=0}+4 (\partial_z v) \partial_y \partial_z^3 u|_{z=0}.
 \end{eqnarray*}
 Consequently we  apply the similar computation as in the proof of Lemma \ref{lemgm} to get,  using Sobolev inequality   and the estimate   \eqref{emix},   
 \begin{eqnarray*}
 \begin{aligned}
 	  \big\|\big(\partial_x^m\partial_z^6 u|_{z=0}\big)\big\|_{L^2_{x,y}}  	& \leq  C\sum_{0\leq j\leq 1}	\big\|\partial_z^j\partial_x^m\big[(\partial_z^3 u)\inner{\partial_x \partial_z u+\partial_y\partial_z v} \big] \big\|_{L^2 } \\
  &\qquad+C\sum_{0\leq j\leq 1}	\big\|\partial_z^j\partial_x^m\big[(\partial_z u) \partial_x \partial_z^3 u \big] \big\|_{L^2 }   \\
  &\qquad+C\sum_{0\leq j\leq 1}	\big\|\partial_z^j\partial_x^m\big[(\partial_z v) \partial_y \partial_z^3 u\big] \big\|_{L^2 }\\
&\leq   C\frac{    [\inner{m+5-7}!]^{\sigma}}{\rho^{ m+5-7}} \abs{\vec a}_{ \rho,\sigma}^2, 
\end{aligned}
 \end{eqnarray*}
 and similarly by \eqref{bo4},
 \begin{eqnarray}\label{ef4}
 	  \big\|\big(\partial_x^m\partial_z^4 u|_{z=0}\big)\big\|_{L^2_{x,y}}  	 \leq   C\frac{    [\inner{m+3-7}!]^{\sigma}}{\rho^{ m+3-7}} \abs{\vec a}_{ \rho,\sigma}^2. 
\end{eqnarray}
 Thus 
 \begin{eqnarray*}
 \begin{aligned}
 	 &\bigg|\int_{\mathbb R^2}  \big( \partial_x^m  \partial_z^5 u\big)\big(   \partial_x^m  \partial_z^6 u\big)\big|_{z=0}dxdy\bigg| \\
 		&\leq \frac{1}{4}\big\|\partial_z \partial_x^m  \partial_z^5 u\big)\big\|_{L^2}^2+C\big\| \partial_x^m  \partial_z^5 u\big)\big\|_{L^2}^2+C   \big\|\big(\partial_x^m\partial_z^6 u|_{z=0}\big)\big\|_{L^2_{x,y}}^2  \\
 &	 \leq \frac{1}{4}\big\|  \partial_x^m  \partial_z^6 u \big\|_{L^2}^2+ C\frac{    [\inner{m+5-7}!]^{2\sigma}}{\rho^{2(m+5-7)}}   \big(\abs{\vec a }_{ \rho,\sigma}^2+   \abs{\vec a }_{ \rho,\sigma}^4 \big)
 \end{aligned}
 \end{eqnarray*} 
 and
 \begin{multline*}
 	 \bigg|\int_{\mathbb R^2}  \big( \partial_x^m  \partial_z^3 u\big)\big(   \partial_x^m  \partial_z^4 u\big)\big|_{z=0}dxdy\bigg| \\
 		 	 \leq \frac{1}{4}\big\|  \partial_x^m  \partial_z^4 u \big\|_{L^2}^2+ C\frac{    [\inner{m+3-7}!]^{2\sigma}}{\rho^{2(m+3-7)}}   \big(\abs{\vec a }_{ \rho,\sigma}^2+   \abs{\vec a }_{ \rho,\sigma}^4 \big).
 \end{multline*} 
 This gives  the validity of  \eqref{eb}  for $j=3, 5$.   Note that \eqref{eb} obviously holds for $j=1, 2$ since $\partial_z^2 u|_{z=0}=0.$  Thus    \eqref{eb}  is valid for $j=1,2,3,5$,  and this   with \eqref{efei} yields 
 \begin{multline}\label{mxe}
   	\norm{\comi z^{\ell+j}  \partial_x^m  \partial_z^j u(t)}_{L^2}^2 +\int_0^t\big\|\partial_z\big(\comi z^{\ell+j}  \partial_x^m  \partial_z^j u(s)\big)\big\|_{L^2}^2ds\\
  \leq  C\frac{    [(m+j-7)!]^{2\sigma} }   {\rho^{2(m+j-7)}}\Big[\norm{(u_0,v_0)}_{2\rho_0,\sigma}^2   +  \int_0^{t}  \big(\abs{\vec a(s)}_{\rho,\sigma}^2+\abs{\vec a(s)}_{\rho,\sigma}^4 \big)ds\Big]
  \end{multline}
  for $j=1, 2,3, 5$.

 It remains to prove the validity of \eqref{mxe} for the case of $j=4$.    By Sobolev inequality, we compute  
 \begin{eqnarray*}
 \begin{aligned}
 	&\Big|\int_{\mathbb R^2} \big(  \partial_x^m  \partial_z^4 u\big)   \big(\partial_x^m  \partial_z^{5} u\big)\big|_{z=0}  dxdy\Big| \\
 &	\leq
 \frac{\rho^2}{m^{2\sigma}}  \big\|\big(\partial_x^m\partial_z^5 u|_{z=0}\big)\big\|_{L^2_{x,y}}^2+  \frac{m^{2\sigma}}{\rho^2}  \big\|\big(\partial_x^m\partial_z^4 u|_{z=0}\big)\big\|_{L^2_{x,y}}^2\\
 &\leq C \frac{\rho^2}{m^{2\sigma}}\Big(  \big\| \partial_x^m\partial_z^5 u \big\|_{L^2}^2+ \big\| \partial_x^m\partial_z^6 u \big\|_{L^2}^2\Big)+  C\frac{    [\inner{m+4-7}!]^{2\sigma}}{\rho^{2( m+4-7)}} \abs{\vec a}_{ \rho,\sigma}^4, \end{aligned}
 \end{eqnarray*} 
 the last inequality using \eqref{ef4}.  This, with \eqref{efei} for $j=4$, yields
 \begin{eqnarray*}
 \begin{aligned}
 &	\norm{\comi z^{\ell+4}  \partial_x^m  \partial_z^4 u(t)}_{L^2}^2+\frac{3}{4}\int_0^t \big\|\partial_z\big(\comi z^{\ell+4}  \partial_x^m  \partial_z^4 u(s)\big)\big\|_{L^2}^2 dt\\
 &\leq\frac{   [\inner{m+4-7}!]^{2\sigma}}{\rho^{2(m+4-7)}}\norm{u_0,v_0}_{2\rho_0,\sigma}^2+ C\frac{\rho^2}{m^{2\sigma}}\int_0^t\Big(  \big\| \partial_x^m\partial_z^5 u(s) \big\|_{L^2}^2+ \big\| \partial_x^m\partial_z^6 u(s) \big\|_{L^2}^2\Big)ds\\
 &\quad+C\frac{   [\inner{m+4-7}!]^{2\sigma}}{\rho^{2(m+4-7)}}\int_0^t   \big(\abs{\vec a(s)}_{ \rho,\sigma}^2+   \abs{\vec a(s)}_{ \rho,\sigma}^4 \big) ds.
 	\end{aligned}
 \end{eqnarray*}
 Moreover   we use \eqref{emix}  to get 
 \begin{eqnarray*}
 	 \frac{\rho^2}{m^{2\sigma}}\int_0^t   \big\| \partial_x^m\partial_z^5 u(s) \big\|_{L^2}^2 ds  \leq C\frac{   [\inner{m+4-7}!]^{2\sigma}}{\rho^{2(m+4-7)}}\int_0^t \abs{\vec a(s)}_{ \rho,\sigma}^2  ds,
 \end{eqnarray*}
 and meanwhile observe that  we have proven   \eqref{mxe} holds for $j=5$ and this implies 
 \begin{eqnarray*}
 \begin{aligned}
 &	\frac{\rho^2}{m^{2\sigma}}\int_0^t  \big\| \partial_x^m\partial_z^6 u(s) \big\|_{L^2}^2 ds\\
 &	\leq 2\frac{\rho^2}{m^{2\sigma}} \int_0^t\big\|\partial_z\big(\comi z^{\ell+5}  \partial_x^m  \partial_z^5 u(s)\big)\big\|_{L^2}^2ds+C\frac{\rho^2}{m^{2\sigma}} \int_0^t\big\| \comi z^{\ell+5}  \partial_x^m  \partial_z^5 u(s) \big\|_{L^2}^2ds\\
 &	\leq  2 \frac{    [(m+4-7)!]^{2\sigma} }   {\rho^{2(m+4-7)}}\norm{(u_0,v_0)}_{2\rho_0,\sigma}^2   \\
 &\quad+C \frac{    [(m+4-7)!]^{2\sigma} }   {\rho^{2(m+4-7)}}   \int_0^{t}  \big(\abs{\vec a(s)}_{\rho,\sigma}^2+\abs{\vec a(s)}_{\rho,\sigma}^4 \big)ds
 		.
 		\end{aligned}
 \end{eqnarray*} 
 Combining the above inequalities we obtain the validity of \eqref{mxe} for $j=4$.   
  The proof of Lemma \ref{lemnor} is thus completed.
 \end{proof}
 
\section{Estimate on   $\partial^\alpha\lambda, \partial^\alpha \delta$ and $\partial^\alpha \tilde\lambda, \partial^\alpha \tilde\delta$}
  \label{seclamdelta}

  Recall  $\lambda, \delta$ and $\tilde\lambda, \tilde\delta$ are the functions given by \eqref{laga},  and 
    this section is devoted to  treating the  terms involving these functions  in the representation of $\abs{\vec a}_{\rho,\sigma}$ (see Definition \ref{gevspace}).    
   
 \begin{proposition}\label{propum} Under Assumption \ref{assmain} we have, 
 	  for any  $t\in[0,T]$ and   for  any pair $\inner{\rho,\tilde\rho}$ with  $0<\rho<\tilde\rho< \rho_0\leq 1$, 
\begin{multline*}
  \sup_{\abs\alpha\geq 6} \frac{\rho^{2( \abs\alpha-6)}}{ [(\abs\alpha-6)!]^{2\sigma}}\Big(\abs\alpha^2\norm{ \partial^\alpha \lambda(t)}_{L^2}^2+\abs\alpha^2\norm{ \partial^\alpha \delta(t)}_{L^2}^2 \Big)  \\  
  +\sup_{\abs\alpha\leq 5}  \Big(\abs\alpha^2\norm{\partial^\alpha \lambda(t)}_{L^2}^2+\abs\alpha^2\norm{ \partial^\alpha \delta(t)}_{L^2}^2 \Big) \\
   \leq C_1\norm{(u_0, v_0)}_{2\rho_0,\sigma}^2+ e^{CC_*^2} \bigg(    \int_0^{t}   \big( \abs{\vec a(s)}_{ \rho,\sigma}^2+\abs{\vec a(s)}_{ \rho,\sigma}^4\big)   \,ds+     \int_0^{t}  \frac{  \abs{\vec a(s)}_{ \tilde\rho,\sigma}^2}{\tilde\rho-\rho}\,ds\bigg), 
\end{multline*}
where $C_*\geq 1$ is the constant given in \eqref{condi1}. 
Similarly the above estimate still holds with $\partial^\alpha \lambda$ and $\partial^\alpha \delta$ replaced  respectively by  $\partial^\alpha \tilde\lambda$ and $\partial^\alpha \tilde\delta$. 
 \end{proposition}

To prove the above proposition we first derive the eqaution solved by  $ \lambda$.   Note that 
  \begin{eqnarray*}
	  \lambda= \partial_x u-(\partial_z u)\int_0^z\mathcal U d\tilde z=\psi_1,
\end{eqnarray*} 
with $\psi_1$ defined by \eqref{plam}.  Then using \eqref{eqp} for $m=1$ we obtain 
  the equation for $\lambda$: 
\begin{multline*}  
\big(\partial_t+u\partial_x+v\partial_y+w\partial_z-\partial_z^2 \big)  \lambda\\
 = -(\partial_xu)^2-(\partial_xv)\partial_yu - \Big[(\partial_y v)\partial_zu- ( \partial_yu)\partial_zv\Big]  \int_0^z\mathcal U d\tilde z+2(\partial_z^2 u)\mathcal U.
\end{multline*}
Now for any $m\geq 6$ we apply $\partial_x^m$ to  the above equation; this gives
\begin{eqnarray*}\label{lam1}
\begin{aligned}
&	\big(\partial_t+u\partial_x+v\partial_y+w\partial_z-\partial_z^2 \big)  \partial_x^m \lambda\\
	& \quad= - \sum_{j=1}^m{m\choose j} \Big[(\partial_x^j u)\partial_x^{m-j+1}\lambda+ (\partial_x^j v)\partial_x^{m-j}\partial_y\lambda+ (\partial_x^j w)\partial_x^{m-j}\partial_z\lambda\Big]\\
	&\qquad\ \ - \partial_x^m\Big[(\partial_xu)^2+(\partial_xv)\partial_yu + \Big[(\partial_y v)\partial_zu- ( \partial_yu)\partial_zv\Big]  \int_0^z\mathcal U d\tilde z-2(\partial_z^2 u)\mathcal U\Big].
	\end{aligned}
\end{eqnarray*}
Thus  taking the scalar product   with $m^2\partial_x^m\lambda$ and observing $\lambda|_{z=0}=0$ and $\lambda|_{t=0}=\partial_xu_0$ we have  
 \begin{equation}\label{k1k2}
 \begin{aligned}
 	&\frac{m^2}{2} \Big(\norm{ \partial_x^m\lambda(t)}_{L^2}^2- \norm{ \partial_x^{m+1}u_0}_{L^2}^2\Big)+m^2 \int_0^{t}	\big\| \partial_z  \partial_x^m\lambda(s)\big\|_{L^2}^2ds  \\
 	 &=  m^2 \int_0^{t}	\inner{\big(\partial_s+u\partial_x+v\partial_y+w\partial_z-\partial_z^2 \big)  \partial_x^m\lambda, \    \partial_x^m\lambda}_{L^2} ds\\
 	 &=   K_1+ K_2+K_3, 
 	 \end{aligned}
 \end{equation}
 where
 \begin{eqnarray*} 
 \begin{aligned}
 K_1=&   - 	m^2\int_0^{t}\Big( \sum_{j=1}^m{m\choose j} \Big[(\partial_x^j u)\partial_x^{m-j+1}\lambda+ (\partial_x^j v)\partial_x^{m-j}\partial_y\lambda \Big], \   \partial_x^m\lambda \Big)_{L^2}ds\\
  & - 	m^2\int_0^{t}\Big( \sum_{j=1}^m{m\choose j}   (\partial_x^j w)\partial_x^{m-j}\partial_z\lambda, \   \partial_x^m\lambda \Big)_{L^2}ds,\\
  K_2	=&- 	m^2\int_0^{t}\big(  \partial_x^m\big[(\partial_xu)^2 +(\partial_xv)\partial_yu + \big[(\partial_y v)\partial_zu- ( \partial_yu)\partial_zv\big]  \int_0^z\mathcal U d\tilde z \big],    \partial_x^m\lambda \big)_{L^2}ds,\\
K_3=&  2m^2\int_0^{t}\Big( \partial_x^m\big[ (\partial_z^2 u)\mathcal U\big], \    \partial_x^m\lambda \Big)_{L^2}ds.
 \end{aligned}
 \end{eqnarray*}
  To estimate the above $K_j, 1\leq j\leq 3,$ we need     the upper bounds of  $\int_0^z \mathcal U dz$ and $\mathcal U$  similar  as  that in \eqref{condi1}, which is stated in the following 
   	
   	\begin{lemma}\label{lemauxi}
   	Under the condition \eqref{condi1} we have, denoting $\partial^\beta=\partial_x^{\beta_1}\partial_y^{\beta_2}$ and recalling $T\leq 1$,   		
   	\begin{eqnarray*}
	\forall\ t\in[0,T],\quad \sum_{\abs\beta\leq 9}\big\|	\comi z^{-\ell}  \int_0^z\partial^{\beta}\mathcal U(t) dz \big\|_{L^2} +\sum_{\stackrel{\abs\beta+j\leq 8}{0\leq j\leq 2}}\big\|	\partial^{\beta}\partial_z^j \mathcal  U(t)   \big\|_{L^2} \leq e^{CC_*^2}, 
\end{eqnarray*}
and
	\begin{eqnarray*}
	\forall\ t\in[0,T],\quad  \sum_{\stackrel{\abs\beta+j\leq 8}{0\leq j\leq 2}}\norm{\partial^\beta\partial_z^j\lambda}_{L^2}   \leq e^{CC_*^2}, 
\end{eqnarray*}
 where $C_*\geq 1$ is the constant in  \eqref{condi1}, and      $C$ is  a constant  depending only on    the Sobolev embedding  constants and the numbers $\rho_0, \sigma, \ell$   given in Definition \ref{defgev}. 
   	\end{lemma} 
    	
   	\begin{proof}
   	This just follows from  direct computation.   Precisely we use standard energy method for the equation \eqref{eqoff} solved by $f= \int_0^z \mathcal U dz$,  applying $   \comi z^{-\ell}\partial^\beta=\comi z^{-\ell}\partial_{x}^{\beta_1}\partial_{y}^{\beta_2},\abs\beta\leq 9,$ to the equation \eqref{eqoff}  and then taking the scalar product with $\comi z^{-\ell}\partial^\beta f$; this with Sobolev inequality  \eqref{soblev} and the condition \eqref{condi1} gives    
 \begin{multline*}
{1\over 2} \frac{d}{dt}\sum_{\abs\beta\leq 9}\big\|\comi z^{-\ell}\partial^\beta f\big\|_{L^2}^2  	+ \sum_{\abs\beta\leq 9}\big\|\partial_z\big(\comi z^{-\ell}\partial^\beta f\big)\big\|_{L^2}^2\\
  \leq  C  C_*^2    \sum_{\abs\beta\leq 9}\big\|\comi z^{-\ell}\partial^\beta f\big\|_{L^2}^2  +CC_*\sum_{\abs\beta\leq 9}\big\|\comi z^{-\ell}\partial^\beta f\big\|_{L^2},
\end{multline*}
where $C_*\geq 1$ is just the constant given in \eqref{condi1}.  Moreover 
we apply again the energy method for the equation \eqref{eqfou} solved by $\mathcal U$,   to obtain
\begin{multline*}
	{1\over 2} \frac{d}{dt}\sum_{\stackrel{\abs\beta+j\leq 8}{0\leq j\leq 2}}\big\|	  \partial^{\beta}\partial_z^j\mathcal U   \big\|_{L^2} ^2  	+ \sum_{\stackrel{\abs\beta+j\leq 8}{0\leq j\leq 2}}\big\|\partial_z\partial^{\beta}\partial_z^j\mathcal U\big\|_{L^2}^2 \\
	 \leq  C  C_*^2    \sum_{\stackrel{\abs\beta+j\leq 8}{0\leq j\leq 2}}\big\|	  \partial^{\beta}\partial_z^j\mathcal U   \big\|_{L^2} ^2+C  C_*    \sum_{\stackrel{\abs\beta+j\leq 8}{0\leq j\leq 2}}\big\|	  \partial^{\beta}\partial_z^j\mathcal U   \big\|_{L^2}    +C C_*\sum_{\abs\beta\leq 9}\big\|\comi z^{-\ell}\partial^\beta f\big\|_{L^2}^2.
\end{multline*}
 As a result using  Gronwall inequality we obtain  the first estimate as desired in Lemma \ref{lemauxi},  which with the representation of $\lambda$ given in  \eqref{laga} as well as \eqref{condi1},  yields the second one.   The proof of Lemma \ref{lemauxi} is completed. 
\end{proof}
 
Now we continue the proof of Proposition \ref{propum}.  Recall $K_1$ is given in  \eqref{k1k2}.   By the second estimate in Lemma \ref{lemauxi} we can apply similar argument for proving Lemma \ref{lemmac} to compute,  using \eqref{elam} here instead of \eqref{uint+} and observing  there is a factor $m$ before   $\norm{\partial_x^m\lambda}_{L^2}$ in \eqref{elam},   
 \begin{multline}\label{ek1}
 	K_1\leq {m^2\over 4} \int_0^{t} \norm{\partial_z\partial_x^{m} \lambda(s)}_{L^2}^2 ds \\
 	+  C\frac{    [\inner{m-6}!]^{2\sigma}}{\rho^{2(m-6)}} \bigg( \int_0^{t}  \big(\abs{\vec a(s)}_{ \rho,\sigma}^3+   \abs{\vec a(s)}_{ \rho,\sigma}^4 \big)ds+e^{CC_*^2}\int_0^{t}   \frac{\abs{\vec a(s)}_{\tilde \rho,\sigma}^2}{\tilde\rho-\rho} ds\bigg)
   	\end{multline}
   	and 
   	\begin{eqnarray} \label{ek2}
 		K_2\leq   C\frac{    [\inner{m-6}!]^{2\sigma}}{\rho^{2(m-6)}}  \bigg(\int_0^{t}   \abs{\vec a(s)}_{ \rho,\sigma}^3 ds
   	 + e^{CC_*^2}  \int_0^{t}   \frac{\abs{\vec a(s)}_{\tilde \rho,\sigma}^2}{\tilde\rho-\rho} ds\bigg).
 	\end{eqnarray}
It remains to treat $K_3$ given in \eqref{k1k2}.   
We
  write
  \begin{eqnarray*}
  	K_3=K_{3,1}+K_{3,2}+K_{3,3},
  \end{eqnarray*}
  with
  \begin{eqnarray*} 
  	 K_{3,1}&=&2m^2\int_0^{t}\Big( \sum_{j=0}^{[m/2]}{m\choose j}  (\partial_x^{j}\partial_z^2  u)\partial_x^{m-j}\mathcal U, \    \partial_x^m\lambda \Big)_{L^2}ds,\\
  	  K_{3,2}&=&2m^2\int_0^{t}\Big( \sum_{j =[m/2]+1}^{m-5}{m\choose j}  (\partial_x^{j}\partial_z^2  u)\partial_x^{m-j}\mathcal U, \    \partial_x^m\lambda \Big)_{L^2}ds,\\
  	 K_{3,3}&=& 2m^2\int_0^{t}\Big( \sum_{j=m-4}^{m}{m\choose j}  (\partial_x^{j}\partial_z^2  u)\partial_x^{m-j}\mathcal U, \    \partial_x^m\lambda \Big)_{L^2}ds.
  \end{eqnarray*} 
  The rest part is devoted to estimating the above terms $K_{3,1}, K_{3,2}$ and $K_{3,3}$.    
  Using \eqref{emix}, \eqref{uint+} and \eqref{elam} as well as \eqref{condi1} and recalling $\sigma\geq 3/2$ we follow the similar computation as in the proof of Lemma \ref{lemmac} to obtain
  \begin{equation*}
  \begin{aligned}
  	K_{3,1}&\leq 2m^2 \Big[\sum_{0\leq j\leq 1}+\sum_{j=2}^{[m/2]}\Big]{m\choose j} \int_0^{t}  \norm{(\partial_x^{j}\partial_z^2   u}_{L^\infty} \norm{\partial_x^{m-j}\mathcal U}_{L^2}  \norm{    \partial_x^m\lambda }_{L^2} ds\\
  	& \leq C C_* \frac{    [\inner{m-6}!]^{2\sigma}}{\rho^{2(m-6)}}  \int_0^{t}   \frac{\abs{\vec a(s)}_{\tilde \rho,\sigma}^2}{\tilde\rho-\rho} ds+  C\frac{    [\inner{m-6}!]^{2\sigma}}{\rho^{2(m-6)}}  \int_0^{t}   \abs{\vec a(s)}_{ \rho,\sigma}^3 ds,
  	\end{aligned}
  \end{equation*}
  and
\begin{eqnarray*}
\begin{aligned}
 K_{3,2}&\leq 2m^2 \sum_{j =[m/2]+1}^{m-5}{m\choose j}    \int_0^{t}\norm{(\partial_x^{j}\partial_z^2   u}_{L^2} \norm{\partial_x^{m-j}\mathcal U}_{L_{x,y}^\infty(L_z^2)}  \norm{    \partial_x^m\lambda }_{L_{x,y}^2(L_z^\infty)} ds \\
 &\leq C  \int_0^{t}\frac{    [\inner{m-6}!]^{\sigma}}{\rho^{m-6}} \abs{\vec a(s)}_{\rho,\sigma}^2\big(m \norm{    \partial_x^m\lambda }_{L^2}+m\norm{   \partial_z \partial_x^m\lambda }_{L^2}\big) ds\\
 &\leq {m^2\over 8} \int_0^{t} \norm{\partial_z\partial_x^{m} \lambda(s)}_{L^2}^2 ds +C\frac{    [\inner{m-6}!]^{2\sigma}}{\rho^{2(m-6)}}  \int_0^{t}  
 \inner{ \abs{\vec a(s)}_{ \rho,\sigma}^3 + \abs{\vec a(s)}_{ \rho,\sigma}^4}ds.
\end{aligned}
\end{eqnarray*}
Finally,
integration by parts gives
  \begin{eqnarray*}
  \begin{aligned}
  	K_{3,3}=& -2m^2\int_0^{t}\Big( \sum_{j=m-4}^{m}{m\choose j}  (\partial_x^{j}\partial_z   u) \partial_x^{m-j}\mathcal U, \    \partial_z\partial_x^m\lambda \Big)_{L^2}ds\\
  	&-2m^2\int_0^{t}\Big( \sum_{j=m-4}^{m}{m\choose j}  (\partial_x^{j}\partial_z   u) \partial_z\partial_x^{m-j}\mathcal U, \    \partial_x^m\lambda \Big)_{L^2}ds\\
  	\leq &  \frac{ m^2}{8}  \int_0^{t}\norm{ \partial_z\partial_x^m\lambda(s)}_{L^2}^2 dt+C\int_0^{t} \bigg[ m\sum_{j=m-4}^{m}{m\choose j}\norm{  (\partial_x^{j}\partial_z   u) \partial_x^{m-j}\mathcal U}_{L^2} \bigg]^2ds \\
  	&+2m^2 \sum_{j=m-4}^{m}{m\choose j}   \int_0^{t}\norm{\partial_x^{j}\partial_z   u}_{L^2} \norm{\partial_z\partial_x^{m-j}\mathcal U}_{L^\infty} \norm{     \partial_x^m\lambda }_{L^2}ds.  \end{aligned}
  \end{eqnarray*}
As for the last term on the right side,   we use \eqref{elam}, \eqref{emix} and the first assertion in  Lemma \ref{lemauxi} to compute
 \begin{multline*}
 	2m^2 \sum_{j=m-4}^{m}{m\choose j}   \int_0^{t}\norm{\partial_x^{j}\partial_z   u}_{L^2} \norm{\partial_z\partial_x^{m-j}\mathcal U}_{L^\infty} \norm{     \partial_x^m\lambda }_{L^2}ds\\
 	\leq e^{CC_*^2}\frac{    [\inner{m-6}!]^{2\sigma}}{\rho^{2(m-6)}}  \int_0^{t}   \frac{\abs{\vec a(s)}_{\tilde \rho,\sigma}^2}{\tilde\rho-\rho} ds. 
 \end{multline*} 
 Meanwhile we claim
 \begin{multline}\label{clam}
 \int_0^{t} \bigg[ m\sum_{j=m-4}^{m}{m\choose j}\norm{ \big(\partial_x^{m-j}\mathcal U\big) \partial_x^{j}\partial_z   u}_{L^2} \bigg]^2ds  \\
\leq e^{CC_*^2}  \frac{  [(m-6)!]^{2\sigma}  } {\rho^{2(m-6)}} \bigg(  \int_0^{t}  \big(\abs{\vec a(s)}_{\rho,\sigma}^2+\abs{\vec a(s)}_{\rho,\sigma}^4 \big)ds
 		+    	\int_0^{t}   \frac{ \abs{\vec a(s)}_{\tilde\rho,\sigma}^2}{\tilde\rho-\rho} ds\bigg).
 \end{multline}
The proof of \eqref{clam} is postponed to the end of this section.  Thus we  combine the above three inequalities to get
\begin{multline*}
  K_{3,3} \leq  \frac{ m^2}{8}  \int_0^{t}\norm{ \partial_z\partial_x^m\lambda}_{L^2}^2 ds  \\
  +e^{CC_*^2}   \frac{    [\inner{m-6}!]^{2\sigma}}{\rho^{2(m-6)}} \Big(   \int_0^{t}  \big(\abs{\vec a(s)}_{\rho,\sigma}^2+\abs{\vec a(s)}_{\rho,\sigma}^4 \big)ds
 		+     	\int_0^{t}   \frac{ \abs{\vec a(s)}_{\tilde\rho,\sigma}^2}{\tilde\rho-\rho} ds\Big),
  \end{multline*}
  which with upper bounds of $K_{3,1}$ and $K_{3,2}$ yields
\begin{multline*}
  K_3 \leq  \frac{ m^2}{4}  \int_0^{t}\norm{ \partial_z\partial_x^m\lambda}_{L^2}^2 ds  \\
  +e^{CC_*^2}   \frac{    [\inner{m-6}!]^{2\sigma}}{\rho^{2(m-6)}} \Big(   \int_0^{t}  \big(\abs{\vec a(s)}_{\rho,\sigma}^2+\abs{\vec a(s)}_{\rho,\sigma}^4 \big)ds
 		+     	\int_0^{t}   \frac{ \abs{\vec a(s)}_{\tilde\rho,\sigma}^2}{\tilde\rho-\rho} ds\Big).
  \end{multline*}
   Now we 
 put the above   estimate and the estimates \eqref{ek1}-\eqref{ek2} on $K_1, K_2 $ into \eqref{k1k2} to obtain
  \begin{multline*}
 	m^2 \norm{ \partial_x^m\lambda(t)}_{L^2}^2+m^2 \int_0^{t}	\big\| \partial_z  \partial_x^m\lambda(s)\big\|_{L^2}^2ds 
 	 \leq m^2 \norm{ \partial_x^{m+1}u_0}_{L^2}^2 \\
 	   + e^{CC_*^2}   \frac{    [\inner{m-6}!]^{2\sigma}}{\rho^{2(m-6)}} \Big(   \int_0^{t}  \big(\abs{\vec a(s)}_{\rho,\sigma}^2+\abs{\vec a(s)}_{\rho,\sigma}^4 \big)ds
 		+     	\int_0^{t}   \frac{ \abs{\vec a(s)}_{\tilde\rho,\sigma}^2}{\tilde\rho-\rho} ds\Big),
 \end{multline*}
which with  the fact that
\begin{eqnarray*}
 m^2\norm{\partial_x^{m+1}u_0}_{L^2}^2 \leq m^2 \frac{[(m-6)!]^{2\sigma}  }{  (2\rho_0)^{2(m-6)}} \norm{(u_0,v_0)}_{2\rho_0,\sigma}^2\leq 4^6 \frac{ [(m-6)!]^{2\sigma}  }{  \rho^{2(m-6)}}\norm{(u_0,v_0)}_{2\rho_0,\sigma}^2
\end{eqnarray*} 
due to the fact  $\rho<\rho_0$, gives the desired upper bound for $\partial_x^m\lambda$ with    
  $m\geq 6$. 
And the estimate for  $m\leq 5$ is straightforward.   
  The above estimate still holds with $\partial_x^m\lambda$ replaced by $\partial_y^m\lambda,$ which can be treated in the same way.  Thus in view of \eqref{realp} the desired estimate on $\partial^\alpha\lambda$ follows.   We can apply the similar argument   to get the   upper bounds of  $\partial^\alpha\delta, \partial^\alpha\tilde\lambda$     and $\partial^\alpha\tilde\delta$. Thus the proof  of  Proposition \ref{propum} will be completed if the assertion \eqref{clam} holds.

  \begin{proof}
  	[Proof of the assertion \eqref{clam}]  We write
  	\begin{multline*}
  	 \int_0^{t} \bigg[ m\sum_{j=m-4}^{m}{m\choose j}\norm{  ( \partial_x^{j}\partial_z   u)  \partial_x^{m-j}\mathcal U}_{L^2} \bigg]^2ds \\
  		 \leq C m^2 \int_0^{t} \norm{ \mathcal U  \partial_x^{m}\partial_z   u}_{L^2}^2ds+C m^4 \int_0^{t} \norm{ \big(\partial_x\mathcal U \big) \partial_x^{m-1}\partial_z   u}_{L^2}^2ds \\
  		 +C \sum_{j=m-4}^{m-2}m^{2(m-j+1)}\int_0^{t} \norm{ \partial_x^{j}\partial_z   u}_{L^2}^2\norm{ \partial_x^{m-j}\mathcal U  }_{L^\infty}^2ds.
  	\end{multline*}
  	As for the last term on the right side, we use  the first inequality in Lemma and the fact that $\sigma\geq 3/2$ to compute directly 
  	  \ref{lemauxi},
  	\begin{multline*}
  		 \sum_{j=m-4}^{m-2}m^{2(m-j+1)}\int_0^{t} \norm{ \partial_x^{j}\partial_z   u}_{L^2}^2\norm{ \partial_x^{m-j}\mathcal U  }_{L^\infty}^2ds\\
  		  \leq e^{CC_*^2}  m^6m^{-4\sigma}  \frac{ [(m-6)!]^{2\sigma}  } {  \rho^{2(m-6)}}\int_0^t\abs{\vec a(s)}^2ds  \leq e^{CC_*^2}\frac{ [(m-6)!]^{2\sigma}  } {  \rho^{2(m-6)}}\int_0^t\abs{\vec a(s)}^2ds.
  	\end{multline*}
  	Thus the desired \eqref{clam} will follow if we can show that 
  \begin{equation}\label{mp}
  \begin{aligned}
	&\int_0^t\norm{\mathcal U \partial_z\partial_x^m u }_{L^2}^2ds\\
	& \leq e^{CC_*^2}  \frac{  [(m-7)!]^{2\sigma}  } {\rho^{2(m-7)}} \Big(  \int_0^{t}  \big(\abs{\vec a(s)}_{\rho,\sigma}^2+\abs{\vec a(s)}_{\rho,\sigma}^4 \big)ds
 		+    	\int_0^{t}   \frac{ \abs{\vec a(s)}_{\tilde\rho,\sigma}^2}{\tilde\rho-\rho} ds\Big)
 		\end{aligned}
 		\end{equation}
 		and
 		\begin{multline}\label{mp1}
	\int_0^t\norm{(\partial_x\mathcal U) \partial_z\partial_x^{m-1} u }_{L^2}^2ds\\
	\leq e^{CC_*^2}  \frac{  [(m-8)!]^{2\sigma}  } {\rho^{2(m-8)}} \Big(  \int_0^{t}  \big(\abs{\vec a(s)}_{\rho,\sigma}^2+\abs{\vec a(s)}_{\rho,\sigma}^4 \big)ds
 		+    	\int_0^{t}   \frac{ \abs{\vec a(s)}_{\tilde\rho,\sigma}^2}{\tilde\rho-\rho} ds\Big).
  \end{multline}
The argument is quite similar as that for proving Lemma \ref{lemtan}. In fact, recalling $\psi_m$ is defined in \eqref{plam} and multiplying both side of \eqref{eqp} by $\mathcal U$ instead of $\comi z^\ell$ therein, we obtain
 \begin{multline*}
\big(\partial_t+u\partial_x+v\partial_y+w\partial_z-\partial_z^2\big)\mathcal U \psi_m  
 		= \mathcal U \inner{F_{m}- L_m}-2(\partial_z\mathcal U)\partial_z\psi_m\\
 		+\Big[\partial_x^2u+\partial_y\partial_x v-(\partial_zu)\partial_x \int_0^z\mathcal U d\tilde z -(\partial_zv)\partial_y \int_0^z\mathcal U d\tilde z +(\partial_xu+\partial_y v)\mathcal U \Big]\psi_m,
 \end{multline*}
 where we used the equation \eqref{eqfou}.   In view of the first assertion in Lemma \ref{lemauxi},  we   repeat  the argument for proving  \eqref{epsm} with slight modification to conclude, observing $\mathcal U\psi_m|_{t=0}=0,$
 \begin{multline*}
   \norm{\mathcal U(t)\psi_m(t)}_{L^2}^2 +\int_0^{t} \big\|\partial_z\big[\mathcal U(s)\psi_m(s)\big]\big\|_{L^2}^2 ds\\
 	\leq e^{CC_*^2}  \frac{   [(m-7)!]^{2\sigma}  }{\rho^{2(m-7)}}	\bigg(\int_0^{t}  \big(\abs{\vec a(s)}_{\rho,\sigma}^2+\abs{\vec a(s)}_{\rho,\sigma}^4 \big)ds 
 		+    \int_0^{t}   \frac{ \abs{\vec a(s)}_{\tilde\rho,\sigma}^2}{\tilde\rho-\rho} ds\bigg).
 \end{multline*}
 On the other hand, 
 \begin{eqnarray*}
 	\begin{aligned}
 		&\int_0^t\norm{\mathcal U \partial_z\partial_x^m u }_{L^2}^2ds\leq   2\int_0^t\norm{\partial_z\big[\mathcal U\partial_x^m u\big]}_{L^2}^2ds+C\int_0^t\norm{\big(\partial_z\mathcal U\big) \partial_x^m u}_{L^2}^2ds \\
 		&  \leq 4\int_0^t\bigg(\norm{\partial_z \big[\mathcal U\psi_m \big]}_{L^2}^2+\big\|\partial_z\big[\mathcal U (\partial_zu)\int_0^z\partial_x^{m-1} \mathcal U d\tilde z\big]\big\|_{L^2} ^2\bigg)ds\\
 		&\quad+C\int_0^t\norm{\big(\partial_z\mathcal U\big) \partial_x^m u}_{L^2}^2ds\\
 		&   \leq 4\int_0^t\norm{\partial_z \big[\mathcal U\psi_m \big]}_{L^2}^2ds+e^{CC_*^2}\int_0^t \norm{\partial_x^{m-1} \mathcal U}_{L^2} ^2ds+e^{CC_*^2}\int_0^t\norm{ \partial_x^m u}_{L^2}^2ds,
 \end{aligned}
 	\end{eqnarray*}
the second inequality following from \eqref{plam} and the last inequality  using Lemma \ref{lemauxi} and the assumption \eqref{condi1}.   Combining the above inequalities we conclude for any $m\geq 7$, using again \eqref{uint+} and \eqref{emix},
\begin{multline*}
	\int_0^t\norm{\mathcal U(s) \partial_z\partial_x^m u(s) }_{L^2}^2ds\\
	\leq e^{CC_*^2}  \frac{  [(m-7)!]^{2\sigma}  } {\rho^{2(m-7)}} \bigg(  \int_0^{t}  \big(\abs{\vec a(s)}_{\rho,\sigma}^2+\abs{\vec a(s)}_{\rho,\sigma}^4 \big)ds
 		+    	\int_0^{t}   \frac{ \abs{\vec a(s)}_{\tilde\rho,\sigma}^2}{\tilde\rho-\rho} ds\bigg).
\end{multline*}
We have proven \eqref{mp}. Similarly for \eqref{mp1}.    Thus the proof of \eqref{clam} is completed.
\end{proof}

\section{Proof of the main result}
\label{sec8}
We will prove in this section the main result on the existence and uniqueness for Prandtl system \eqref{prandtl}.     Since the proof is  similar as in 2D case once we
 have the a priori estimate,  we will only  give  a sketch, and refer to \cite[Section 7 and Section 8]{LY} for the detailed discussion.

\begin{proof}[Proof of Theorem \ref{maithm1}]
	 
	The proof  relies on the a priori estimates given in Theorems \ref{apriori}.  In order to  obtain  the existence  of solutions to   the   Prandtl equations \eqref{prandtl},  there are  two main 
ingredients,  and one is to investigate the existence of   approximate solutions to the regularized Prandtl system
\begin{equation}
\label{repradtl}
\left\{
\begin{aligned}
	&\partial_t u_\eps+\big(u_\eps \partial_x  +v_\eps\partial_y	 +w_\eps\partial_z\big)  u_\eps-\partial _{z}^2u_\eps-\eps\partial _{x}^2 u_\eps-\eps\partial _{y}^2 u_\eps=0,\\
 &\partial_t v_\eps+\big(u_\eps \partial_x  +v_\eps\partial_y	 +w_\eps\partial_z\big) v_\eps-\partial _{z}^2v_\eps-\eps\partial _{x}^2 v_\eps-\eps\partial _{y}^2 v_\eps=0,\\
 &(u_\eps,v_\eps)|_{z=0}=(0,0),\quad\lim_{z\rightarrow +\infty}   (u_\eps,v_\eps)=(0, 0),\\
 &(u_\eps,v_\eps)|_{t=0}=(u_0,v_0),
 \end{aligned}
 \right.
\end{equation}
with $w_\eps=-\int_0^z(\partial_x u_\eps+\partial_y v_\eps )d\tilde z.$       We remark that  the  regularized equations above share the same compatibility condition   \eqref{comcon} as the  original system  \eqref{prandtl}.  Another ingredient is to  derive a uniform estimate with respect to $\eps$ for the approximate solutions $(u_\eps, v_\eps).$

  The existence for the parabolic system \eqref{repradtl} is standard.  Indeed, 
suppose that  $(u_0,v_0)\in X_{2\rho_0,\sigma}.$  Then we can  construct,   following  the 
  similar scheme as that in \cite[Section 7]{LY},  a solution  $(u_\eps,v_\eps)\in L^\infty\big([0, \widetilde T_\eps];X_{3\rho_0/2, \sigma}\big)$ to \eqref{repradtl}    for some  $\widetilde T_\eps>0$ that may depend on $\eps$.  

It remains to  derive a uniform estimate for the approximate solutions $(u_\eps, v_\eps),$  so that we can remove the $\eps$-dependence of the lifespan $\widetilde T_\eps.$     
To do so we  define as in Subsection \ref{subaux}  the auxilliary fucntions $\mathcal U_\eps, \lambda_\eps, \delta_\eps$ in the similar way as  that for $\mathcal U, \lambda, \delta$  given in Subsection \ref{subaux},  with $(u,v,w)$ and the Prandtl operator therein replaced respectively by $(u_\eps, v_\eps, w_\eps)$  and the regularized Prandtl operator given above.  Similarly for $\widetilde{\mathcal U}_\eps, \tilde\lambda_\eps, \tilde\delta_\eps $.    Accordingly denote  \begin{eqnarray*}
	\vec a_\eps=(u_\eps, v_\eps,  \mathcal U_\eps, \widetilde{\mathcal U}_\eps, \lambda_\eps, \tilde \lambda_\eps, \delta_\eps, \tilde\delta_\eps) 
\end{eqnarray*}
and 
  define $\abs{\vec a_\eps}_{\rho,\sigma}$  similarly as that of  $\abs{\vec a}_{\rho,\sigma}$  (see Definition \ref{gevspace}).    Note that
  $$\vec a_{\eps}|_{t=0}=\inner{u_0,v_0, 0, 0, \partial_x u_0, \partial_y u_0,\partial_xv_0,\partial_yv_0}.$$
Then  
   we can verify directly that      
   \begin{eqnarray}\label{remarkdif}
   \forall\ \rho\leq \rho_0,\quad 	  \abs{\vec a_\eps(0)}_{\rho, \sigma}  \leq C_{\rho_0,\sigma}\norm{(u_0,v_0)}_{2\rho_0,\sigma}, 
   \end{eqnarray}
   with $C_{\rho_0,\sigma}$ a constant depending only on $\rho_0$ and $\sigma$.

   Let $\tau>1$ be a fixed number to be determined later. We define  
 \begin{equation}\label{trinormdef}
 \normm{\vec a_\eps}_{(\tau)}\stackrel{\rm def}{ =}\sup_{\rho, t} \inner{ \frac{ \rho_0-\rho- \tau t}{\rho_0-\rho}  }^{1/2}\abs{\vec a_\eps(t)}_{\rho,\sigma},
\end{equation}
 where the supremum is taken over all pairs $(\rho, t)$ such that $ \rho>0,\, 0\leq t \leq  \rho_0/(4\tau)$ and  $\rho+ \tau t <\rho_0$.  
 Leting $C_{\rho_0,\sigma}$ be the constant given in   \eqref{remarkdif}  and letting  $C_1\geq 1$ be the constant given in Theorem \ref{apriori} which depends only on $\rho_0,\sigma$ and the Sobolev embedding constants,  we denote
   \begin{eqnarray}\label{cstar}
 	C_0= 2\inner{C_{\rho,\sigma} +C_1} \norm{(u_0,v_0)}_{2\rho_0,\sigma}+1. 
 \end{eqnarray}
In the following discussion, we will use the bootstrap argument to prove the    assertion that  
 \begin{equation}\label{ac}
   \normm{\vec a_\eps}_{(\tau)} \leq C_0 /2	
 \end{equation}
for some $\tau$ large enough,  if the  condition  
 \begin{eqnarray}\label{ab}
  	 \normm{\vec a_\eps}_{(\tau)} \leq   C_0 
  \end{eqnarray} 
  is fulfilled.    
 
 \medskip
\noindent {\it Step 1)}.  Observe,  for any   $t\in  [0, \rho_0/(4\tau)]$,   
\begin{multline}\label{les}
\frac{\sqrt 2}{2}\norm{(u_\eps(t), v_\eps(t))}_{{\rho_0\over 2},\sigma}\leq \frac{\sqrt 2}{2}\abs{\vec a_\eps(t)}_{{\rho_0\over2},\sigma}\\
 \leq   \inner{ \frac{ \rho_0-{\rho_0\over2} - \tau t}{\rho_0-{\rho_0\over 2} }  }^{1/2}\abs{\vec a_\eps(t)}_{{\rho_0\over 2},\sigma}\leq \normm{\vec a_\eps}_{\inner{\tau}} .
\end{multline}  
Thus
under the condition \eqref{ab}, we have $\norm{(u_\eps(t), v_\eps(t))}_{\rho_0/2,\sigma}\leq \sqrt 2 C_0$ for any    $t\in [0,  \ \rho_0/(4\tau)]$ and thus it follows from  
 the  definition of $\norm{(u_\eps(t),v_\eps(t))}_{\rho_0/2,\sigma}$ that, for any $t\in [0,  \ \rho_0/(4\tau)]$,
\begin{equation*} 
 	  \sup_{\stackrel{0\leq j\leq 5}{\abs\alpha+j\leq 10}}  \Big(\big\|\comi z^{\ell+j} \partial^\alpha \partial_z^j   u_\eps(t)\big\|_{L^2}+\big\|\comi z^{\ell+j} \partial^\alpha \partial_z^j  v_\eps(t)\big\|_{L^2}\Big) \leq   \tilde C_{\rho_0,\sigma} C_0
\end{equation*}
with $\tilde C_{\rho_0,\sigma}\geq 1$ a constant depending only on $\rho_0$ and $\sigma$. 
 Then Assumption \ref{assmain} is fulfilled with $(u,v)$ therein replaced by $(u_\eps, v_\eps)$ and thus similar to Theorem \ref{apriori}  we  can  repeat the argument in Sections \ref{sec5}-\ref{seclamdelta} with minor modification to obtain  the following assertion:  for any $t\in [0,  \ \rho_0/(4\tau)]$  and   any pair $(\rho,\tilde\rho)$ with $0<\rho<\tilde\rho<\rho_0\leq 1,$
\begin{multline}\label{aee}
	\abs{\vec a_\eps (t)}_{\rho,\sigma}^2 
	\leq C_1  \norm{(u_0, v_0)}_{2\rho_0, \sigma}^2 \\
	+e^{C_2C_0^2}\bigg( \int_{0}^{t} \inner{\abs{\vec a_\eps(s)}_{\rho,\sigma}^2+\abs{\vec a_\eps(s)}_{\rho,\sigma}^4} \,ds 
	+  \int_{0}^{t}\frac{ \abs{\vec a_\eps(s)}_{\tilde\rho,\sigma}^2}{\tilde \rho-\rho}\,ds\bigg),
\end{multline}
where $C_2>0$ is a constant depending only on the numbers $\rho_0, \sigma$ and the Sobolev embedding constants but independent of $\eps$, and  the constant $C_1\geq 1$  is just the one  given in Theorem \ref{apriori}.  

\medskip
\noindent{\it Step 2)}.  
We let 
   $(\rho, t)$ be an arbitrary pair which is fixed at moment and satisfies that  $ \rho>0,\,  t\in[0, \rho_0/(4\tau)]$ and  $\rho+ \tau t<\rho_0.$
 Then it follows from the definition  \eqref{trinormdef}  of  $\normm{\vec a_\eps}_{(\tau) }$ that
   \begin{equation}\label{fiesonssd}
  \forall~0\leq s\leq  t,  \quad 	 \abs{\vec a_\eps(s)}_{\rho,\sigma}\leq  \normm{\vec a_\eps}_{(\tau)} \inner{ \frac{\rho_0-\rho} { \rho_0-\rho- \tau s} }^{1/2}.
   \end{equation}
     Furthermore,
we take in particular  such a $\tilde\rho(s)$ that 
\begin{eqnarray*}
\tilde \rho(s)=\frac{\rho_0+  \rho- \tau  s}{2}.
\end{eqnarray*}
Then direct calculation shows 
that 
\begin{eqnarray}\label{rhos}
 \forall~0\leq s\leq t ,\qquad \rho< \tilde \rho(s)  \quad{\rm and}\quad \tilde \rho(s)+ \tau  s <\rho_0,
\end{eqnarray}  
and
 \begin{eqnarray}\label{rhom}
 \forall~0\leq s\leq t ,\qquad \tilde\rho(s)-\rho=\frac{\rho_0-\rho - \tau  s}{2}=  \rho_0-\tilde \rho(s)- \tau s.
\end{eqnarray}
By the  inequalities in \eqref{rhos} and the second equality in \eqref{rhom} it follows  that,  for any $0\leq s\leq t,$
\begin{equation}\label{tilrho}
\begin{aligned}
\abs{\vec a_\eps(s)}_{\tilde\rho(s),\sigma}\leq \normm{\vec a_\eps}_{(\tau )}\inner{\frac{\rho_0-\tilde\rho(s)}{\rho_0-\tilde\rho(s)- \tau s}}^{1\over 2}\leq \normm{\vec a_\eps}_{(\tau )}\inner{\frac{2\inner{ \rho_0- \rho}}{\rho_0- \rho- \tau s}}^{ 1\over 2}.
\end{aligned}
\end{equation} 
Putting     \eqref{fiesonssd}  and  \eqref{tilrho}  into     the estimate \eqref{aee}  and using  the first equality in \eqref{rhom},  we have 
\begin{eqnarray*}
\begin{aligned}
 \abs{ \vec a_\eps(t)}_{ \rho,\sigma}^2\leq&  C_1 \norm{(u_0, v_0)}_{2\rho_0,\sigma}^2 +e^{C_2C_0^2}\normm{\vec a_\eps}_{(\tau )}^2\int_0^{t}   \frac{\rho_0-\rho  }{ \rho_0-\rho - \tau s} \,ds\\
& +e^{C_2C_0^2}\normm{\vec a_\eps}_{(\tau)}^2\Big(  \normm{\vec a_\eps}_{(\tau )}^2\int_0^{t}   \frac{\inner{\rho_0-\rho}^2}{\inner{\rho_0-\rho - \tau  s}^2}  ds +   \int_0^{t }   \frac{2^2 \inner{\rho_0-\rho}}{\inner{\rho_0-\rho - \tau s}^2}  ds\Big)\\
     \leq&  C_1  \norm{(u_0, v_0)}_{2\rho_0,\sigma}^2  +  \frac{e^{C_2C_0^2}(5+C_0^2)}{\tau} \normm{\vec a_\eps}_{(\tau)}^2 \frac{\rho_0-\rho} {\rho_0-\rho-\tau t},
   \end{aligned}
\end{eqnarray*}
where in the last inequality we have used the condition 
 \eqref{ab} and the fact that 
\begin{eqnarray*}
	  \frac{\rho_0-\rho  }{ \rho_0-\rho - \tau  s} \leq  \frac{\inner{\rho_0-\rho}^2}{\inner{\rho_0-\rho - \tau  s}^2}\leq \frac{    \rho_0-\rho  }{\inner{\rho_0-\rho - \tau  s}^2}.
\end{eqnarray*}
Thus we multiply both sides by the fact $\inner{\rho_0-\rho-\tau t}/\inner{\rho_0-\rho}$ and observe   $(\rho, t)$ is an arbitrary pair with $\rho>0$, $t\in[0, \rho_0/(4\tau)]$  and $\rho+\tau t< \rho_0$; this with $C_1\geq 1$ gives   
\begin{equation}\label{eslas}
	\normm{\vec a_\eps}_{(\tau)}  \leq   C_1  \norm{(u_0, v_0)}_{2\rho_0,\sigma}  
	+ \frac{ \sqrt{e^{C_2C_0^2}(5+C_0^2) }}{\sqrt{\tau}}\normm{\vec a_\eps}_{(\tau)}.
\end{equation}
Now we choose such a $\tau$ that 
\begin{eqnarray}\label{taua}
 	1-\frac{ \sqrt{e^{C_2C_0^2}(5+C_0^2) }}{\sqrt{\tau}}= \frac{C_1}{C_1+ C_{\rho,\sigma}}.
\end{eqnarray}
Then it follows from \eqref{eslas} that 
\begin{eqnarray*}
	\normm{\vec a_\eps}_{(\tau)}\leq \inner{C_1+ C_{\rho,\sigma}} \norm{(u_0, v_0)}_{2\rho_0, \sigma} \leq C_0/2,
\end{eqnarray*} 
recalling  $C_0$ is given by \eqref{cstar}.  This gives the desired assertion \eqref{ac} provided \eqref{ab} holds.  Thus by  the bootstrap argument  we conclude, with $\tau$ defined by \eqref{taua},  
\begin{eqnarray*}
	\normm{\vec a_\eps}_{(\tau)}\leq \inner{C_1+ C_{\rho,\sigma}} \norm{(u_0, v_0)}_{2\rho_0, \sigma}+1/2,
\end{eqnarray*} 
which with \eqref{les} yields 
\begin{eqnarray*}
	\forall \ t\in[0,   \rho_0/(4\tau)],\quad 
	\norm{(u_\eps(t), v_\eps(t))}_{\rho_0/ 2,\sigma}\leq \sqrt 2\inner{C_1+ C_{\rho,\sigma}} \norm{(u_0, v_0)}_{2\rho_0, \sigma} +\frac{\sqrt 2}{2}.
\end{eqnarray*}
Now letting  $\eps\rightarrow 0$ we have, by compactness arguments,  the limit $u$ of $u_\eps$ solves the equation \eqref{prandtl}.  We complete the existence part of Theorem \ref{maithm1}.    The uniqueness will follow from a similar argument as in \cite[Subsection 8.2]{LY} so we omit it here for brevity. Thus the proof of  Theorem \ref{maithm1} is completed.  
\end{proof}

   \bigskip
\noindent {\bf Acknowledgements.}
 The research of the first author was supported by NSFC (Nos. 11871054, 11771342, 11961160716) and  the Natural Science Foundation of Hubei Province (2019CFA007).  The second author  is partially  supported  by 
SITE (Center for Stability, Instability, and Turbulence at NYUAD)  and  by NSF-DMS 1716466.
The research of the third
author was  supported by the General Research Fund of Hong Kong, CityU No.11302518 and the  Fundamental Research Funds for the Central Universities No.2019CDJCYJ001.

        








 \frenchspacing
\bibliographystyle{cpam}


\begin{thebibliography}{99}
\bibitem{awxy}  Alexandre, R.; Wang, Y.-G.; Xu, C.-J.; Yang, T. Well-posedness of the Prandtl equation in Sobolev spaces. {\it J. Amer. Math. Soc.} {\bf 28} (2015), no. 3, 745-784.


\bibitem{asa}
Asano, K. A note on the abstract Cauchy-Kowalewski theorem. {\it Proc. Japan Acad. Ser. A Math. Sci}. {\bf 64} (1988), no. 4, 102-105.
 
 \bibitem{Bardos-Titi}
 Bardos, Claude W.; Titi, Edriss S. Mathematics and turbulence: where do we stand? {\it J. Turbul. } {\bf 14} (2013), no. 3,, 42-76.
 
 \bibitem{Bardos}
Bardos, Claude W.; Titi, Edriss S.;  Wiedemann E.     Onsager's conjecture with physical boundaries and an application to the vanishing viscosity limit.   Preprint, arXiv:1803.04939v1.

 
 


\bibitem{CGIM} 
Collot, C.; Ghoul, T. E.; Ibrahim, S.; and  Masmoudi, N.   On singularity formation for the two dimensional unsteady Prandtl's system.  Preprint (2018),  arXiv:1808.05967.

\bibitem{CGM} 
Collot, C.; Ghoul, T. E;    Masmoudi, N.    Unsteady separation for the inviscid two-dimensional Prandtl's system.  Preprint(2019),  arXiv:1903.08244.
 
\bibitem{cons}
  Constantin, P. Note on loss of regularity for solutions of the 3-D incompressible Euler and related equations. {\it Comm. Math. Phys.} {\bf  104} (1986), no. 2, 311-326. 
  
 
 \bibitem{D-M}
Dalibard, A.-L.;    Masmoudi, N.,      Separation for the stationary Prandtl equation. Preprint(2018), arXiv:1802.04039v1.
 
 \bibitem{DG}  
 Dietert, H.; G\'erard-Varet, D.  Well-posedness of the Prandtl equations without any structural assumption. {\it  Ann. PDE} {\bf 5} (2019), no. 1, Paper No. 8, 51 pp. 
 
 \bibitem{e-2}  E, W.;  Engquist, B.  Blowup of solutions of the unsteady Prandtl's equation. {\it Comm. Pure Appl. Math.} {\bf 50} (1997), no. 12, 1287-1293.

 \bibitem{ZZ-2}   Fei, M.;   Tao, T.;   Zhang, Z.-F.     On the zero-viscosity limit of the Navier-Stokes equations in the half-space.  Preprint, arXiv:1609.03778.  
 

 
\bibitem{GV-D}  G\'erard-Varet, D.;  Dormy, E.  On the ill-posedness of the Prandtl equation. {\it J. Amer. Math. Soc. }{\bf 23}  (2010), no. 2, 591-609.
  
  \bibitem{GM}
   G\'{e}rard-Varet, D.;   Maekawa, Y.  Sobolev stability of Prandtl expansions for the steady Navier-Stokes equations. {\it Arch. Ration. Mech. Anal.} {\bf  233} (2019), no. 3, 1319-1382.
 
 \bibitem{GMM}
 G\'{e}rard-Varet, D.;   Maekawa, Y.;  Masmoudi, N. Gevrey stability of Prandtl expansions for 2-dimensional Navier-Stokes flows. {\it Duke Math. J.}  {\bf 167} (2018), no. 13, 2531-2631.


\bibitem{GM}
 G\'{e}rard-Varet,  D.;     Masmoudi,  N.
Well-posedness for the Prandtl system without analyticity or monotonicity. {\it Ann. Sci. \'Ec. Norm. Sup\'er.} (4) {\bf 48} (2015), no. 6, 1273-1325.
 
 
  
  
 \bibitem{GV-N} G\'erard-Varet, D.; Nguyen, T. Remarks on the ill-posedness of the Prandtl equation. {\it Asymptot. Anal.} {\bf 77} (2012), no. 1-2,  71-88.
 
 \bibitem{grenier} Grenier, E. On the nonlinear instability of Euler and Prandtl equations. {\it Comm. Pure Appl. Math.} {\bf  53} (2000), no. 9, 1067-1091.
 
\bibitem{GGN} Grenier, E.; Guo, Y.; Nguyen, Toan T.  Spectral instability of general symmetric shear flows in a two-dimensional channel. {\it Adv. Math. }. {\bf 292} (2016),  52-110.
 
 \bibitem{G-S} 
 Guo, Y.;  Iyer, S.   Validity of steady Prandtl layer expansions.  Preprint,  arXiv:1805.05891V1.
 
\bibitem{guo}  Guo, Y.; Nguyen, T.  A note on Prandtl boundary layers. {\it Comm. Pure Appl. Math.} {\bf 64}  (2011), no. 10,1416-1438.
 
\bibitem{GN2}
 Guo, Y.;  Nguyen, Toan T.  Prandtl boundary layer expansions of steady Navier-Stokes flows over a moving plate. {\it  Ann. PDE} {\bf 3} (2017), no. 1, Paper No. 10, 58 pp.



 
  
 
  \bibitem{LY}
  Li, W.-X.; Yang, T.  Well-posedness in Gevrey function spaces for the Prandtl equations with non-degenerate critical points. {\it J. Eur. Math. Soc. (JEMS)} {\bf 22} (2020), no. 3,  717-775. 

 \bibitem{LY3D}
 Li, W.-X.; Yang, T.
  Well-posedness in Gevrey function space for the three-dimensional Prandtl equations.  Preprint,   arXiv:1708.08217. 
  
 \bibitem{LWY2}
Liu, C.-J.; Wang, Y.-G.; Yang, T. On the ill-posedness of the Prandtl equations in three-dimensional space. {\it Arch. Ration. Mech. Anal.}  {\bf 220} (2016), no. 1, 83-108.
  
 
\bibitem{Mae}
Maekawa, Y. On the inviscid limit problem of the vorticity equations for viscous incompressible flows in the half-plane. {\it Comm. Pure Appl. Math.} {\bf  67} (2014), no. 7, 1045-1128.

\bibitem{mamoudi1}
 Masmoudi, N. Remarks about the inviscid limit of the Navier-Stokes system. {\it Comm. Math. Phys. } {\bf 270} (2007), no. 3, 777-788.


\bibitem{mamoudi}
  Masmoudi, N. {\it  Examples of singular limits in hydrodynamics.} Handbook of differential equations: evolutionary equations. Vol. III, 195-275, Handb. Differ. Equ., Elsevier/North-Holland, Amsterdam, 2007. 
 
 \bibitem{MW}
  Masmoudi, N.; Wong, T.  Local-in-time existence and uniqueness of solutions to the Prandtl equations by energy methods. {\it Comm. Pure Appl. Math.}   {\bf 68}  (2015), no. 10, 1683-1741.

 \bibitem{oleinik-3}  Oleinik, O. A.; Samokhin, V. N.  {\it Mathematical models in boundary layer theory.}  Applied Mathematics and Mathematical Computation, 15. Chapman \& Hall/CRC, Boca Raton, FL, 1999. 


\bibitem{prandtl} Prandtl,  L.      \"{U}ber Fl\"ussigkeitsbewegungen bei sehr kleiner
Reibung.  {\it Verhandlungen des III. Internationalen Mathematiker-Kongresses (Heidelberg 1904)}, 484-491, 
Teubner, 1905

\bibitem{Samm}   Sammartino, M;   Caflisch, R. E.  Zero viscosity limit for analytic solutions, of the Navier-Stokes equation on a half-space. I. Existence for Euler and Prandtl equations. {\it Comm. Math. Phys. } {\bf 192} (1998), no. 2, 433-461;   II. Construction of the Navier-Stokes solution. {\it Comm. Math. Phys. } {\bf 192} (1998), no. 2,  463-491.


\bibitem{xin-zhang} Xin, Z.; Zhang, L.  On the global existence of solutions to the Prandtl's system. {\it Adv. Math. } {\bf 181} (2004), no. 1, 88-133.
\end{thebibliography}

\end{document}